\documentclass{amsart}	
\usepackage{tikz-cd}
\usepackage{amssymb, amsfonts, amsmath, mathrsfs, mathtools}
\usepackage{amsthm, thmtools}
\DeclareMathOperator{\lcm}{LCM}
\renewcommand{\gcd}{\operatorname{GCD}}
\newcommand{\Aut}{\operatorname{Aut}}
\newcommand{\StabAut}{\operatorname{Aut}^{(\infty)}}
\newcommand{\Sym}{\operatorname{Sym}}
\newcommand{\Eig}{\mathfrak{Eig}}
\newcommand{\CAP}{\operatorname{CAP}}
\usepackage{dsfont}
\makeatletter
\@for\cs:={A,B,C,D,E,F,G,H,I,J,K,L,M,N,O,P,Q,R,S,T,U,V,W,X,Y,Z}\do{
  \expandafter\edef\csname c\cs\endcsname{\noexpand\mathcal{\cs}}
}
\@for\cs:={C,N,Q,R,Z}\do{
  \expandafter\edef\csname \cs\endcsname{\noexpand\mathbb{\cs}}
}
\makeatother

\declaretheorem[numberwithin = section]{theorem}
\declaretheorem[sibling = theorem, style = definition]{definition}
\declaretheorem[sibling = theorem]{lemma}
\declaretheorem[sibling = theorem]{proposition}
\declaretheorem[sibling = theorem]{remark}
\declaretheorem[sibling = theorem]{corollary}
\usepackage[shortlabels]{enumitem}
\usepackage[hidelinks]{hyperref}

\usepackage{biblatex}
\numberwithin{equation}{section}  %Equation numbering

\title[Eigenvalues and stabilized automorphism groups]{Eigenvalues and the stabilized automorphism group}
\author{Bastián Espinoza}
\author{Jennifer N. Jones-Baro}

\begin{document}
\maketitle
\vspace{-1cm}
\begin{abstract}
We study the stabilized automorphism group of minimal and, more generally, certain transitive dynamical systems. 
Our approach involves developing new algebraic tools to extract information about the rational eigenvalues of these systems from their stabilized automorphism groups. 
In particular, we prove that if two minimal system have isomorphic stabilized automorphism groups and each has at least one non-trivial rational eigenvalue, then the systems have the same rational eigenvalues. 
Using these tools, we also extend Schmieding's result on the recovery of entropy from the stabilized automorphism group to include irreducible shifts of finite type.

% We develop new algebraic tools to extract information about the rational eigenvalues of these system from their stabilized automorphism groups. 
% Notably, we prove that for any two minimal systems with at least one rational eigenvalue, an isomorphism of stabilized automorphism groups implies the systems have the same rational eigenvalues. 
% We additionally use these tools to extend previously known results to more general settings.
\end{abstract}

%\tableofcontents

%%% issues:
% 1. fix autoref for \item of enumerate

\section{Introduction}
Consider a topological dynamical system $(X, T)$, that is, $X$ is a compact metric space and $T$ is a homeomorphism from $X$ onto itself. 
A primary way to study these systems is by examining their symmetries, or \emph{automorphisms}. 
An automorphism of $(X, T)$ is a homeomorphism $\varphi\colon X\to X$ that commutes with $T$, meaning $\varphi \circ T = T \circ \varphi$. 
The collection of all such automorphisms forms a group, where the group operation is composition. 
This group is known as the {\em automorphism group} and is typically denoted by $\Aut(X,T)$, or simply $\Aut(T)$ if $X$ is clear from context.

% The study of automorphism groups in symbolic systems originated with Hedlund \cite{Hedlund} and has since been extensively explored. 
% Particularly rich research has been carried out for subshifts of finite type. 
% In this direction, central works include the paper by Boyle, Lind, and Rudolph \cite{BLR}, as well as the notable papers by Kim and Roush \cite{K&R} and by Ryan \cite{Ryan}.
% More recent advancements in describing automorphism groups focus on classes of systems with restrictions on word complexity \cite{CyrKraexponential, CyrKralinear, ddmp}.
Suppose $\cA$ is a finite alphabet and $X\subseteq \cA^\Z$ is a closed set that is invariant under the left shift $\sigma\colon \cA^\Z\to \cA^\Z$. 
The system $(\cA^\Z, \sigma)$ is the \emph{full shift} over $\cA$, and $(X,\sigma)$ is a \emph{subshift}. 
The study of automorphism groups of subshifts can be traced back to the seminal work of Hedlund \cite{Hedlund}, and has since expanded significantly.
The case of \emph{subshifts of finite type} (the class characterized by a finite set of forbidden words) is well studied; see for instance the influential paper of Boyle, Lind, and Rudolph \cite{BLR}, and also the works by Kim and Roush \cite{K&R, KR2} and Ryan \cite{Ryan, Ryan2}. 
More recently, there has been a focus on characterizing automorphism groups of symbolic systems at the opposite end of complexity, as seen in the works of Cyr and Kra \cite{CyrKraexponential, CyrKralinear}, and Donoso {\em et al.} \cite{ddmp}.

%One basic problem in the study of automorphism groups is distinguishing between the groups of two given systems.
A basic and thorny problem in this study is to distinguish the automorphism groups of different systems, based on their algebraic structure.
A fundamental question that remains open for most $n,m\geq 1$ is the following: 
if $(X_m, \sigma_m )$ and $(X_n, \sigma_n )$ are the full shifts on $m$ and $n$ letters, respectively, such that $\Aut(\sigma_m)$ is isomorphic as a group to $\Aut(\sigma_n)$, must $m$ be equal to $n$? 
Thus far, the only tool that has provided non-trivial answers to this question is Ryan's Theorem \cite{Ryan, Ryan2}, which gives as an immediate corollary, for example, that $\Aut(\sigma_p)$ is not isomorphic to $\Aut(\sigma_{p^2})$ for all $p\geq 2$. 

Hartman, Kra, and Schmieding, introduced in \cite{stab} the concept of the \textit{stabilized automorphism group}, an extended automorphism group that is defined as
$$ \Aut^{(\infty)}(X,T) = \bigcup_{n=1}^\infty \Aut(X,T^n),$$
and denoted by $\Aut^{(\infty)}(T)$ when $X$ is clear from context.
%They succeeded in distinguishing certain stabilized automorphism groups, including those of the full shifts on 2 and 6 letters, a problem whose non-stabilized counterpart remains open.
%More precisely, they showed that for the full shift on $n$ letters, the number of distinct prime divisors of $n$ is an isomorphism invariant of the stabilized automorphism group.
They showed that for the full shift on $n$ letters, the number of distinct prime divisors of $n$ is an isomorphism invariant of the stabilized automorphism group. In contrast to the open question for the non-stabilized setting, this implies for instance that $\Aut^{(\infty)}(\sigma_2)$ is not isomorphic to $\Aut^{(\infty)}(\sigma_6)$.
% This was later expanded by Schmieding \cite{Pent}, who showed that the topological entropy of a mixing subshift of finite type can be retrieved (up to a rational multiplicative constant) from its stabilized automorphism group, allowing him to distinguish the stabilized automorphism group of full shifts on $m$ and $n$ symbols for $m$ and $n$ multiplicatively independent.
%This was later extended by Schmieding \cite{Pent}, who showed that the topological entropy of a mixing subshift of finite type can be retrieved (up to a rational multiplicative constant) from its stabilized automorphism group, allowing him to completely characterize when the stabilized automorphism group of two full shifts are isomorphic.
This was later extended by Schmieding \cite{Pent}, who showed that $\Aut^{(\infty)}(\sigma_n) $ is isomorphic to $ \Aut^{(\infty)}(\sigma_m)$ if and only if $\tfrac{\log n}{\log m}\in \Q$. 
His novel technique allowed to show that the topological entropy of a mixing subshift of finite type is recoverable from its stabilized automorphism group. 
This motivates the question for other classes of shifts: what kind of dynamical information does the stabilized automorphism group contain?

The second author explored this question for odometers and Toeplitz subshifts satisfying certain hypotheses \cite{jones}.
Since odometers have zero entropy, the type of dynamical information that can be recovered from the stabilized automorphism group differs from the high complexity case of shift of finite type.
The main results in \cite{jones} can be stated in terms of rational eigenvalues. 
A \textit{rational eigenvalue} of a system $(X,T)$ is a complex number $e^{2\pi i r}$, with $r\in\Q,$ for which there exists a continuous function $g\colon X\to\C$, not identically zero, such that $g(T(x)) = e^{2\pi i r}\cdot g(x)$ for all $x\in X$. 
Define
\begin{equation*}
    \Eig(T) = \{ q \geq 1 : \text{$\exp(2\pi i\cdot 1/q)$ is a rational eigenvalue of $(X,T)$} \}.
\end{equation*} 
This set is never empty since $1$ is an eigenvalue for any system.
We extend the results in \cite{jones} to cover much broader classes of systems:
%It was proved in \cite{jones} that if two torsion-free odometers or Toeplitz subshifts with torsion-free maximal equicontinuous factors have isomorphic stabilized automorphism groups, then they have the same rational eigenvalues. 
% To do so, they studied the semidirect product structures that arise in $\Aut(X,T^n)$ and compare it to those in $\Aut(X,T^m)$ for $m<n$. 
% Our first main result is the following. 
\begin{theorem} \label{theo:recover_eigs:abstract}
Let $(X,T)$ and $(Y,S)$ be minimal systems, each having at least one rational eigenvalue other than $1$.
If $\StabAut(X,T)$ is isomorphic to $\StabAut(Y,S)$, then $(X,T)$ and $(Y,S)$ have the same rational eigenvalues.
\end{theorem}
Both odometers and Toeplitz subshifts are covered by \autoref{theo:recover_eigs:abstract}. 
As a corollary, we remove the torsion-free hypotheses in the main results in \cite{jones}.

We conjecture that the hypothesis on rational eigenvalues different from 1 in \autoref{theo:recover_eigs:abstract} is necessary. 
Addressing this question, we show that this hypothesis can be removed when adding an extra assumption on $\StabAut(X,T)$.
To precisely state this assumption, define the set
\begin{equation*}
    M(T) = \{n \geq 1: \text{$T^n$ acts transitively on $X$}\}.
\end{equation*} 
See Section \ref{section:backgroud} for details on transitivity.
% Our second main result is the following.

\begin{theorem} \label{theo:recover_eigs:low_comp:abstract}
Let $(X,T)$ and $(Y,S)$ be minimal systems. Assume each of  \newline $\bigcup_{n\in M(T)}\Aut(X,T^n)$ and $\bigcup_{n\in M(T)}\Aut(Y,S^n)$ is either abelian or virtually $\Z$. 
If $\StabAut(X,T)$ is isomorphic to $\StabAut(Y,S)$, then $(X,T)$ and $(Y,S)$ have the same rational eigenvalues.
\end{theorem}
The previous theorem covers, for instance, any minimal subshift whose word-complexity function grows non-superlinearly.
The proofs of Theorems \ref{theo:recover_eigs:abstract} and \ref{theo:recover_eigs:low_comp:abstract} are completed in Section \ref{sec:recover_eigs}.
\medskip

% The reader is referred to Section \ref{section:backgroud} for the precise definitions and notation that are used in the discussion that follows.
The following result illustrates the role played by rational eigenvalues in the algebraic structure of the stabilized automorphism group. It is proved in Section \ref{section:rationaleigenvalues}.
 
\begin{theorem} \label{theorem:abstract:wreath}
Let $(X,T)$ be a transitive system, $m\in \Eig(T)$, and $n \geq 1$ be such that $T^n$ acts transitively on $X$.
Then there is a $T^m$-invariant clopen subset $X_m\subseteq X$, such that
\begin{equation*}
    \Aut(T^k) \cong \Aut(T^{nm}|_{X_m})^m \rtimes \Sym(m),
\end{equation*} 
where $\Sym(m)$ is the symmetric group on $m$ symbols.
% Moreover, if $M(T)$ is the set of of $n \geq 1$ such that $T^n$ acts transitively on $X$, then 
% \begin{equation*}
%     \varinjlim_{m \in \Eig(T)}
%     \bigcup_{n\in M(T)} \Aut(T^{m\cdot n}|_{X_m})^m \rtimes \Sym(m).
% \end{equation*}
\end{theorem}
This theorem is proven in Section \ref{section:rationaleigenvalues} and a description of the clopen subset $X_m$ is provided in \autoref{eigenpartition}. We also characterize the semidirect product in the previous theorem as a \textit{wreath product}.
See Subsection \ref{subsection:wreath} for the definition and Section \ref{section:algebra} for more details on wreath products.

Under certain conditions, we are able to completely describe the stabilized automorphism group of a system. This involves defining \textit{cyclic almost partitions} associated to $(X,T)$, meaning finite partitions of $X$ with certain dynamical properties. 
Briefly, the existence of a cyclic almost partition with $q$ elements weakens the condition for $\exp(2\pi i \cdot 1/ q)$ to be an eigenvalue. The exact definition is left to Section \ref{section:rationaleigenvalues}.
% We defer the precise definition to Section \ref{section:rationaleigenvalues}, but the existence of a cyclic almost partition with $q$ elements generalizes $q \in \Eig(T)$.
We denote by $\operatorname{CAP}(T)$ the set of integers $q\geq 0$ such that  $(X,T)$ has a cyclic almost partition of size $q$. 
For the class of transitive systems for which $\CAP(T)=\Eig(T)$, we extend the statement of \autoref{theorem:abstract:wreath} as follows. 
 
\begin{theorem} \label{theorem:abstract:directLimit}
Let $(X,T)$ be a transitive system satisfying $ \CAP(T)=\Eig(T)$.
Then, 
\begin{equation*}
    \StabAut(T) \cong
    \varinjlim_{m \in \Eig(T)}
    \bigcup_{n\in M(T)} \Aut(T^{m\cdot n}|_{X_m})^m \rtimes \Sym(m),
\end{equation*}
\end{theorem}

The condition $ \CAP(T)=\Eig(T)$ is satisfied by a large subset of dynamical systems, including all minimal systems and irreducible sofic subshifts. This is proved in \autoref{minimal_sofic=>cP}. In Section \ref{subsec:rec_eigs:trans-case}, we provide several examples to illustrate why this hypothesis is necessary in \autoref{theorem:abstract:directLimit}, and moreover, that without this hypothesis, the connection between the stabilized automorphism group and rational eigenvalues is completely lost.

Inspired by the known succinct descriptions of the automorphism group for low complexity systems \cite{CyrKralinear, CyrKraexponential, ddmp}, we consider in Section \ref{section:FiniteAS} minimal systems with finitely many asymptotic classes. 
In this setting, we can give a more precise description of the stabilized automorphism group. 
As a corollary, we obtain that the stabilized automorphism group of a minimal system that has finitely many asymptotic classes and finitely many rational eigenvalues is virtually $\Z^d$, for some $d \geq 1$.

In order to prove Theorems \ref{theo:recover_eigs:abstract} and \ref{theo:recover_eigs:low_comp:abstract} we develop some algebraic tools. These tools are described and proved in Section \ref{section:algebra}. One key result is that if $G$ and $H$ are groups, $n,m \geq 2$, and we have isomorphic wreath products $G \rtimes \Sym(n)$ and $H \rtimes \Sym(m)$, then $n = m$. 

Finally, in Section \ref{section:irreducibleSFT}, we apply our techniques to subshifts of finite type.
Using Schmieding's main result in \cite{Pent}, we prove that it is possible to weaken the mixing condition in his theorem to irreducibility, which is a hypothesis that cannot be further relaxed. 

\begin{theorem} \label{main:irreducible_sofic_case}
Let $(X, \sigma_X )$, $(Y, \sigma_Y)$ be irreducible subshifts of finite type with isomorphic stabilized automorphism groups.
% Suppose that $\{-1, \exp(2\pi i / 3)\}$ are not the only eigenvalues of $(X, \sigma_X )$ and of $(Y, \sigma_Y )$.
Then,
$$ \frac{h_{\mathrm{top}}(X, \sigma_X)}{h_{\mathrm{top}}(Y,\sigma_Y)} \in \Q. $$
\end{theorem}

%\subsection{Guide to this article.}  

\subsection{Acknowledgments}
The authors deeply appreciate the valuable discussions and feedback received throughout this project from Bryna Kra, Alejandro Maass and Scott Schmieding. 
Additionally, they would like to express their gratitude to Solly Coles and Sebastián Donoso for their insightful comments. 
They are also thankful to the Center of Mathematical Modeling (CMM) in Santiago, Chile, for providing funding during their tenure as a visiting researchers.
The second author is also grateful for the support of the National Science Foundation Graduate Research Fellowship under Grant No. DGE-1842165.

\section{Background} \label{section:backgroud}
  
A \textit{topological dynamical system} (or simply a \textit{system}) is a pair $(X,T)$, where $X$ is a compact metric space and $T\colon X\to X$ is a homeomorphism. Given a system $(X,T)$ and a point $x\in X$, we 
define the \textit{orbit of $x$} as $\mathcal{O}_T(x)=\{T^n(x): n\in \Z\}$. 
% Given a subset $U\subseteq X$, we define $\mathcal{O}_T(U)=\bigcup\limits_{x\in U}\mathcal{O}_T(x)$. 
A system is \textit{transitive} if there exists a point $\hat{x}\in X$ such that $\mathcal{O}_T(\hat{x})$ is dense in $X$; in this case, $\hat{x}$ is called a \textit{transitive point}. 
If every $x\in X$ is a transitive point, we call $(X,T)$ a \textit{minimal system}. 
A \textit{minimal set} for $(X,T)$ is a closed $T$-invariant subset $E\subseteq X$ with no proper closed $T$-invariant subset. 
It is a well known fact that $(X,T)$ is minimal if and only if the only minimal set in $(X,T)$ is $X$.
We say that $T$ {\em transitive} (resp.\ {\em minimal}) on $E$ if $E$ is a closed $T$-invariant subset of $X$ such that the restriction $T|_E\colon E \to E$ is transitive (resp.\ minimal).

Let $(X,T)$ and $(Y,S)$ be systems.
If there is a continuous bijective map $\pi\colon X\to Y$ such that $\pi \circ T = S \circ \pi$, then we say that $(X,T)$ and $(Y,S)$ are {\em conjugate} and that $\pi$ is a conjugacy.
A self conjugacy of $(X,T)$ is an {\em automorphism}, and the set $\Aut(X,T)$ of all such conjugacies is called the {\em automorphism group} of $(X,T)$, as it is a group with the composition of functions as the operation.
When the space $X$ is clear from the context, we shorten the notation to $\Aut(T) = \Aut(X,T)$. 

In \cite{stab}, Hartman, Kra and Schmieding introduced the {\em stabilized automorphism group} of a system, that we now define.
Let $(X,T)$ be a system.
Note that, for any $n \geq 1$, $(X,T^n)$ is also a system.
Thus, we can define the stabilized automorphism group as
\begin{equation*}
    \StabAut(X,T) = \bigcup_{n=1}^\infty \Aut(T^n),
\end{equation*}
where the union is taken in the space $\operatorname{Homeo}(X)$ of all homeomorphisms of $X$. 
It clear from the definitions that this is indeed a group. 
When the context is clear, we write $\StabAut(T)$ instead of $\StabAut(X,T)$.

\subsection{Symbolic dynamics}

An {\em alphabet} is a finite set $\cA$ endowed with the discrete topology.
We denote by $\cA^\Z$ the set of bi-infinite sequences in $\cA$ and, for $x \in \cA^\Z$, we denote the value of $x$ at $n\in \Z$ by $x(n)$. 
Equipped with the product topology, $\cA^\Z$ is a compact metrizable space called the full-shift.
A concrete metric that is compatible with the topology of $\cA^\Z$ is $d(x,y) = 2^{-\min\{|i|: x(i) \neq y(i)\}}$, $x,y\in \cA^\Z$.
The map $T\colon \cA^\Z\to \cA^\Z$, defined as $T((x_n)_{n\in Z})(x_{n+1})_{n\in\Z}$ for all $(x_n)\in \cA^\Z$ is called the \textit{left shift}.
Note that $T$ is a homeomorphism of $\cA^\Z$ onto itself. 
If $X$ is a closed $T$-invariant subset of $\cA^\Z$, then the system $(X, T|_{X})$ is called a \textit{subshift} or a \textit{symbolic system}. 
When the context is clear, we identify $T|_{X}$ with $T$ and denote the subshift simply by $(X,T)$. 
When $T$ is clear from context we may also say that $X$ itself is a subshift. % inglés

For all $n\geq 1$, we call an $n$-tuple $w = w_0 \dots w_{n-1} \in \cA^n$ a \textit{word of length $n$}.
For any word $w$ of length $n$, we define \textit{the cylinder set given by $w$} as as 
$$ [w] = \{x \in \cA^\Z : x_i = w_i \text{ for all } 0 \leq i < n \}. $$
The collection of cylinder sets $\{T^i([w]) : w \in \cA^\ast, i \in\Z\}$, where $ \cA^\ast=\bigcup\limits_{j=1}^{\infty}\cA^j$, is a basis for the topology of $\cA^\Z$.

The \textit{language} of a subshift $X$ is % if the language depends on the pair (X,T) (and not just on X), then we would have to write \cL(X,T), which is too cumbersome.
\begin{equation*}
    \cL(X) = \{ w \in \cA^\ast : [w] \cap X \neq \emptyset\}.
\end{equation*}
For $n \geq 1$, let $\cL_n(X)$ be the set of words of length $n$ in $\cL(X)$. 
The \textit{complexity of a subshift} is the map $p_X\colon \N\to\N$ defined as $p_X(n) = |\cL_n(X)|$.

\subsection{Algebraic background} \label{subsection:wreath}

We denote by $\Sym(n)$ the symmetric group on $\Omega = \{0,1,\dots,n-1\}$, where $n \geq 1$.
Note that $\Sym(n)$ acts canonically on $\Omega$ by permuting its elements.
Let $G$ be a group. 
The {\em wreath product of $G$ with $\Sym(n)$} denoted by $G \wr_\Omega \Sym(n)$ is the semi-direct product $G^n\rtimes \Sym(n)$ given by the right-action $\Gamma\colon \Sym(n)\to \Aut(G^n)$ defined as
\begin{equation*}
    \Gamma(\sigma)(g_0,g_1,...,g_{n-1})=(g_{\sigma^{-1}(0)},g_{\sigma^{-1}(1)},...,g_{\sigma^{-1}(n-1)})
\end{equation*}
for $\sigma\in \Sym(n)$ and $(g_0,...,g_{n-1})\in G^n$.
To simplify our notation, we write $G \wr \Sym(n) = G \wr_\Omega \Sym(n)$ and $g_\sigma = \Gamma(\sigma)(g)$ for $g\in G^n$.

Wreath products can be characterized in terms of exact sequences as follows.
A group $G$ is isomorphic to the wreath product $H \wr \Sym(n)$ if and only if there is a split exact sequence 
\begin{equation} 
\begin{tikzcd}[column sep=3ex]
    1 \arrow{r} & H^rn
    \arrow{r}{\psi} & G
    \arrow{r}{\pi} & \Sym(n) \arrow{r} 
    \arrow[bend left=33]{l}{\rho} & 1,
\end{tikzcd}
\end{equation}
such that $\rho(\sigma)^{-1} \psi(h) \rho(\sigma) = \psi(h_\sigma)$ for every $\sigma \in \Sym(n)$ and $h = (h_0,\dots,h_{n-1}) \in H^n$.

Let $G$ be a group and $H$ and $K$ be subgroups of $G$.
We denote by $C_H(K)$ the centralizer subgroup of $K$ in $H$, that is, $C_H(K)$ consists of those elements $h \in H$ that commute with every $k \in K$. If the context is clear, we write $C(K) = C_G(K)$.

% % % % % % % % % % SECTION % % % % % % % % % %
\section{Rational eigenvalues and the stabilized automorphism group} \label{section:rationaleigenvalues}
In this section, we describe the effect of the rational spectrum of a system in the algebraic structure of its stabilized automorphism group.
As we aim to use these results in \autoref{main:irreducible_sofic_case}, with both minimal subshifts and irreducible subshifts of finite type, we introduce the abstract condition, show that it covers these two cases and prove the results of this section assuming only this condition.

% The section is organized as follows.
% We start with the necessary definitions and the introduction of cyclic almost-partition and Condition \eqref{condition_cP}.
% Then, we prove the lemmas describing the relation between eigenvalues, cyclic almost-partitions and the structure of $\StabAut(T)$, which culminates with \autoref{aut_as_eig&min_parts}.
\medskip

An \textit{eigenvalue} of a system $(X,T)$ is a complex number $\lambda$ of modulus 1 for which there exists a continuous function $g\colon X\to\C$, not identically zero, such that $g(T(x)) = \lambda\cdot g(x)$ for all $x\in X$. 
It is a standard fact that $|g|$ is constant if $T$ acts transitively on $X$.
For an eigenvalue $\lambda$, the associated function $g$ is called an \textit{eigenfunction}. 
An eigenvalue is {\em rational} if it has the form $\exp(2\pi i\cdot \alpha)$, for some $\alpha \in \Q$.
Denote 
\begin{equation}
    \Eig(T) = \{ q \geq 1 : \text{$\exp(2\pi i\cdot 1/q)$ is a rational eigenvalue of $(X,T)$} \}.
\end{equation}
Remark that the least common divisor $\lcm(p,q)$ of $p,q \in \Eig(T)$ belongs to $\Eig(T)$, and that $r \in \Eig(T)$ for every $r$ dividing $p \in \Eig(T)$.
% This set is called the \textit{additive spectrum} of $(X,T)$. 

In transitive systems, \autoref{eigenpartition} characterizes rational eigenvalues in terms of cyclic partitions.
Although this result is part of the folklore, we include a proof as the authors could not find any reference to it.

% % % % % % % % % % PROPOSITION % % % % % % % % % %
\begin{proposition} \label{eigenpartition}
Let $(X,T)$ be a transitive dynamical system and $m \geq 1$.
Then, $m \in \Eig(T)$ if and only if there is a $T^m$-invariant, clopen set $X_m$ such that $(T^k X_m : 0 \leq k < m)$ is a partition of $X$.
\end{proposition}
\begin{proof}
Assume that $X_m \subseteq X$ is a $T^m$-invariant clopen set and that $(T^k X_m : 0 \leq k < m)$ is a partition of $X$.
Then, we can define a map $g\colon X \to \C$ by $g(x) = \exp(2\pi i k / m)$ for $x \in T^k X_m$ and $k\in\{0,1,\dots,m-1\}$. 
Notice that, since $T$ is a homeomorphism, we have that every $T^k X_m$ is clopen.
Hence, $g$ is a continuous function. 
Additionally, $g(T(x)) = \exp(2\pi i/m) g(x)$ by construction. 
We conclude that $g$ is an eigenfunction with eigenvalue $\exp(2\pi i/m)$ and that $m \in \Eig(T)$.

Let $\hat x$ be a transitive point in $X$ and let us now suppose that $m \in \Eig(T)$. Take $g\colon X \to \C$ to be an eigenfunction for $\exp(2\pi i/m)$.
Observe that the formula $g(T(\hat{x})) = \exp(2\pi i/m) g(\hat{x})$ implies that $g(\cO_T(\hat{x}))$ consists only of the $m$ points $\{\exp(2\pi i k/m) g(\hat{x}): 0\leq k<m\}$. 
Hence, since $g$ is continuous and $\mathcal{O}_T(\hat{x})$ is dense in $X$, we have that
\begin{equation} \label{eq:eigenpartition:1}
    g(X) = \{\exp(2\pi i k/m) g(\hat{x}): 0 \leq k < m\}.
\end{equation}
Define $X_m = g^{-1}(\{g(\hat{x})\})$.
Notice that the union $X_m \cup T(X_m) \cup \dots \cup T^{m-1}(X_m)$ is disjoint.
Moreover, since $\hat{x}$ is transitive for $T$, we have that $X_m \cup T(X_m) \cup \dots \cup T^{m-1}(X_m)$ is equal to $X$.
Therefore, $\{T^k X_m : 0 \leq k < m\}$ is a partition of $X$ with $m$ elements.
In particular, $X_m$ is clopen.
\end{proof}
% % % % % % % % % % REMARK % % % % % % % % % %
\begin{remark}
The partition defined in \autoref{eigenpartition} is not uniquely determined, as it depends on a choice of the transitive point $\hat{x}$.
\end{remark}

The last result motivates the following definition.
\begin{definition} \label{almost-eigenvalues}
Let $(X,T)$ be a transitive system, $q \geq 1$ and $\tilde{X} \subseteq X$ be a $T^q$-invariant closed set.
We say that $\tilde{X}$ defines a {\em cyclic partition} of size $q$ if $(T^k \tilde{X} : 0 \leq k < q)$ forms a partition of $X$.
It is said that $\tilde{X}$ defines a {\em cyclic almost-partition} of size $q$ if
\begin{enumerate}
    \item $X = \tilde{X} \cup T(\tilde{X}) \cup \dots \cup T^{q-1}(\tilde{X})$, and 
    \item $T^i(\tilde{X}) \cap T^j(\tilde{X})$ has empty interior for all $i,j\in\{0,\dots,q-1\}$ with $i \neq j$.
\end{enumerate}
Lastly, $(X,T)$ has a cyclic partition (resp.\ cyclic almost-partition) if there is a $T^q$-invariant closed set $\tilde{X}$ defining a cyclic partition (resp.\ cyclic almost-partition).
\end{definition}
Remark that, by \autoref{eigenpartition}, a transitive system $(X,T)$ has a cyclic partition of size $m$ if and only if $m \in \Eig(T)$.

For a transitive system $(X,T)$, we write
\begin{equation*}
    \CAP(T) = \{q \geq 1 : \text{$(X,T)$ has a cyclic almost-partition of size $q$}\}.
\end{equation*}
Observe that \autoref{eigenpartition} ensures that, for transitive systems, $\Eig(T)$ is a subset of $\CAP(T)$.
Systems satisfying $\Eig(T) = \CAP(T)$ play an special role in our results.

% Thus, it is natural to consider the following condition:
% \begin{multline} 
%     \label{condition_cP} \tag{$\Lambda$}
%     \text{A transitive system $(X,T)$ satisfies $(\Lambda)$ if 
%     $\exp(2\pi i/q)$ is an eigenvalue
%     } \\ \text{
%     whenever $(X,T)$ has a cyclic almost-partition of size $q$.}
% \end{multline}

% Then, a transitive system $(X,T)$ satisfies \eqref{condition_cP} if and only if $\CAP(T) = \Eig(T)$.

\subsection{Properties of cyclic almost-partitions}

% % % % % % % % % % LEMMA % % % % % % % % % %
\begin{lemma} \label{Peig:props}
Let $(X,T)$ be a transitive system.
\begin{enumerate}
    \item \label{Peig:props:lattice}
    If $p, q \in \CAP(T)$ and $r \geq 1$ divides $p$, then $\lcm(p,q) \in \CAP(T)$ and $r \in \CAP(T)$.
    \item \label{Peig:props:induced}
    Let $q \in \CAP(T)$, $X_q$ be the associated $T^q$-invariant closed set, and $p \geq 1$.
    Then, $p \in \CAP(T^q|_{X_q})$ if and only if $pq \in \CAP(T)$. 
\end{enumerate}
\end{lemma}
\begin{proof}
We fix elements $p$ and $q$ of $\CAP(T)$ and take closed set $X_p$ and $X_q$ that define cyclic almost-partitions of sizes $p$ and $q$, respectively.
As $X_p$ and $X_q$ have non-empty interior, Item (1) in \autoref{almost-eigenvalues} lets us assume without loss of generality that $Y \coloneqq X_p \cap X_q$ has non-empty interior.
Note that $Y$ is a closed $T^k$-invariant set, where $k = \lcm(q,p)$.
Also, by Item (2) in \autoref{almost-eigenvalues},
\begin{equation*}
    X = Y \cup T(Y) \cup \dots \cup T^{k-1}(Y).
\end{equation*}
We claim that if $i,j \in \Z$ are such that $Y_{i,j} \coloneqq T^i(Y) \cap T^j(Y)$ has non-empty interior, then $i = j \pmod{k}$.
Indeed, if $Y_{i,j}$ has non-empty interior, then the interiors of $T^i(X_q) \cap T^j(X_q)$ and $T^i(X_p) \cap T^j(X_p)$ are non-empty, which implies that $i = j \pmod{q}$ and $i = j \pmod{p}$.
Since $k = \lcm(q,p)$, this gives that $i = j \pmod{k}$.

We conclude that 
\begin{equation} \label{eq:Peig:props:1}
    \text{$\lcm(q,p) \in \CAP(T)$ for all $p, q \in \CAP(T)$.}
\end{equation}
Moreover, as $Y \subseteq X_q$, we have that $X_q = Y \cup T^q(X_q) \cup (T^q)^2(X_q) \cup \dots \cup (T^q)^{k/q}(X_q)$ and that the pairwise intersections $(T^q)^i(X_q) \cap (T^q)^j(X_q)$ have empty interior for all $i,j \in \{0,\dots,k/q\}$ with $i \neq j$.
So,
\begin{equation*}
    \text{$\lcm(p,q)/q = k/q \in \CAP(T^q|_{X_q})$ for every $p, q \in \CAP(T)$.}
\end{equation*}
This implies that $r \in \CAP(T^q|_{X_q})$ whenever $rq \in \CAP(T)$, proving the first implications of Part \ref{Peig:props:induced}.

Conversely, assume that $r \in \CAP(T^q_{X_q})$ and let $X_r \subseteq X_q$ be a $T^q_{X_q}$-invariant closed set defining a cyclic almost-partition of size $r$.
Then, by the definition of $X_q$, $X = \bigcup_{0 \leq k < qr} T^k(X_r)$ and 
\begin{equation*}
    T^i(X_r) \cap T^j(X_r) \subseteq \left(\bigcup_{0 \leq k < \ell < r} T^k(X_r)\cap T^\ell(X_r)\right) \cup\left( \bigcup_{0 \leq k < \ell < q} T^k(X_q) \cap T^\ell(X_q)\right).
\end{equation*}
Since the intersections $T^k(X_r)\cap T^\ell(X_r)$ and $T^k(X_q) \cap T^\ell(X_q)$ have empty interior, we deduce that $qr \in \CAP(T)$.

Finally, we prove Part \ref{Peig:props:lattice}.
Let $p, q \in \CAP(T)$ and $r \geq 1$ dividing $p$.
Then, $\lcm(q,p) \in \CAP(T)$ by \eqref{eq:Peig:props:1}.
Also, if $X_p$ defines a cyclic almost-partition of size $p$, then, by considering 
$$ X_r = X_p \cup T^r(X_p) \cup T^{2r} \cup \dots \cup T^{(p/r-1)r}(X_p), $$ 
it can be shown that $r \in \CAP(T)$.
\end{proof}

The next lemma is a partial analogue of \autoref{Peig:props} for cyclic partitions.
It is used in Section \ref{sec:recover_eigs}.
\begin{lemma} \label{CycPart2divisor}
Let $(X,T)$ be a transitive system and $p, q \in \Eig(T)$ be such that $p$ divides $q$.
If $X_q$ is a $T^q$-invariant clopen set defining a cyclic partition, then $X_q \cup T(X_q) \cup \dots \cup T^{q/p-1}(X_q)$ is a $T^p$-invariant closed set defining a cyclic partition of size $p$.
\end{lemma}
\begin{proof}
It is clear that $T^p$ acts transitively on the $T^p$-invariant set $Y = \bigcup_{k\in\Z} T^{pk}(X_q)$.
Also, as $X_q$ is $T^q$-invariant,
\begin{equation*}
    Y = X_q \cup T^p(X_q) \cup T^{2p}(X_q) \cup \dots \cup T^{(q/p-1)p}(X_q).
\end{equation*}
This implies that $Y$ is a union of finitely many clopen sets, and thus that it is clopen.
Moreover, since $(T^k X_q : 0 \leq k < q)$ forms a partition of $X$, we have that $(T^k Y : 0 \leq k < p)$ forms a partition of $X$.
\end{proof}

We now show that the class of systems $(X,T)$ satisfying $\Eig(T) = \CAP(T)$ include any minimal subshift and irreducible subshift of finite type.
% % % % % % % % % % LEMMA % % % % % % % % % %
\begin{lemma} \label{minimal_sofic=>cP}
% If $(X,T)$ is minimal or an irreducible subshift of finite type, then $(X,T)$ has a cyclic almost-partition of size $q$ if and only if $1/q$ is a rational eigenvalue.
If $(X,T)$ is minimal or an irreducible subshift of finite type, then $\Eig(T) = \CAP(T)$.
\end{lemma}
\begin{proof}
Assume first that $(X,T)$ is minimal system and let $q \in \CAP(T)$.
In view of \autoref{Peig:props}, it is enough to show that $q \in \Eig(T)$.
Since $q \in \CAP(T)$, there is a $T_q$-invariant closed set $X_q$ satisfying $X = X_q \cup T(X_q) \cup \dots \cup T_{q-1}(X_q)$ and that $T^i(X_q) \cap T^j(X_q)$ has empty interior for all $i,j\in\{0,\dots,q-1\}$ with $i \neq j$.
Let $i,j\in\{0,\dots,q-1\}$ be different.
We prove that $T^i(X_q) \cap T^j(X_q)$ is empty.
This would imply, by \autoref{eigenpartition}, that $q \in \Eig(T)$, as desired.

Suppose, with the aim to obtain a contradiction, that $Y = T^i(X_q) \cap T^j(X_q)$ is not empty.
Let $Y_\infty \coloneqq \bigcup_{k\in\Z} T^k(Y)$ and note that
\begin{equation} \label{eq:minimal=>almosteigs=eigs:1}
    \text{$Y_\infty$ has empty interior.}
\end{equation}
Now, $Y$ is $T^q$-invariant, so
\begin{equation*}
    Y_\infty \coloneqq \bigcup_{k\in\Z} T^k(Y) = Y \cup T(Y) \cup \dots \cup T^{q-1}(Y).
\end{equation*}
In particular, $Y_\infty$ is a non-empty closed $T^q$-invariant set.
Being $(X,T)$ minimal, this implies that $Y_\infty = X$.
This contradicts \eqref{eq:minimal=>almosteigs=eigs:1}; therefore, $Y$ is empty.
We conclude that the hypothesis of \autoref{eigenpartition} is satisfied, and thus that $q \in \Eig(T)$.
\medskip

Next, we assume that $(X,T)$ is an irreducible subshift of finite type.
It is a well-known fact (see for example \cite[Corollary 4.5.7]{putnam_notes}) that there is a positive integer $p \geq 1$, called the period of $X$, and a $T^p$-invariant clopen set $\tilde{X}$ such that $X$ is partitioned as $\tilde{X} \cup T(\tilde{X}) \cup \dots \cup T^{p-1}(\tilde{X})$ and $(\tilde{X}, T^p|_{\tilde{X}})$ is a mixing subshift of finite type.
In particular,
\begin{equation} \label{eq:minimal=>almosteigs=eigs:sofic_case}
    \text{$T^p|_{\tilde{X}}$ does not have a cyclic almost-partition of size different from $1$.}
\end{equation}
Let $q \in \CAP(T)$ be arbitrary and set $k = \lcm(q,p)$.
Then, \autoref{Peig:props} ensures that $k \in \CAP(T)$ and that $k/p \in \CAP(T^p|_{\tilde{X}})$.
Hence, by \eqref{eq:minimal=>almosteigs=eigs:sofic_case}, $p/k = 1$, which implies that $q$ divides $p$.
Therefore, $q \in \Eig(T)$ by \autoref{Peig:props}.
It follows that $\Eig(T) = \CAP(T)$.
\end{proof}

% % % % % % % % % % PROPOSITION % % % % % % % % % %
\begin{proposition}\label{transpowers}
Let $(X,T)$ be a transitive system and $m \geq 1$. 
Then, the system $(X,T^m)$ is transitive if and only if $\gcd(m,n) = 1$ for all $n \in \CAP(T)$.
\end{proposition}
\begin{proof}
Suppose that $T^m$ is transitive and let $n \in \CAP(T)$.
We have to show that $\gcd(m,n) = 1$.
The definition of $n$ permits to find a closed $T^n$-invariant set $X_n$ such that 
\begin{equation} \label{eq:transpowers:defi_Xn}
    \text{$X = X_n \cup \dots \cup T^{m-1}(X_n)$ and $T^i(X_n) \cap T^j(X_n)$ has empty interior}
\end{equation}
for all pair $i,j \in \Z$ with $i \not= j \pmod{n}$.
Observe that $X_n$ has non-empty interior.
Hence, $Y = \bigcup_{k\in\Z} T^{km} X_n$ is a $T^m$-invariant set with non-empty interior.
Moreover, since the union defining $Y$ is finite (as $X_n$ is $T^n$-invariant), $Y$ is closed.
It follows from the transitivity of $T^m$ that $Y = X$.
We deduce that there is $\ell \in \Z$ such that $T^{\ell m} (X_n) \cap T(X_n)$ has non-empty interior.
Thus, by \eqref{eq:transpowers:defi_Xn}, $\ell m = 1 \pmod{n}$ and $\gcd(m,n) = 1$.

Let us now assume that $\gcd(m,n) = 1$ for all $n \in \CAP(T)$.
Take $\hat{x}$ to be a transitive point  for $T$, and define $Y$ as the closure of $\cO_{T^m}(\hat{x})$.
Set $F = \{k \in \Z : T^k(Y) \subseteq Y\}$.
It is clear that $F$ is an additive subgroup of $\Z$.
Then, the least positive integer $\ell \geq 1$ in $F$ generates $F$ (as an additive group) and divides every element of $F$.
We define $Z = \bigcup_{k\in\Z} T^{k\ell}(Y)$.
Then, since $Z$ is $T^\ell$-invariant, the set
\begin{equation*}
    Z' \coloneqq \bigcup_{k\in\Z} T^k(Z) = Z \cup T(Z) \cup \dots \cup T^{\ell-1}(Z)
\end{equation*}
is closed (as the union is finite) and $T$-invariant.
From this and that $\hat{x} \in Z'$ we deduce that $Z' = X$, that is,
\begin{equation*}
    Z \cup T(Z) \cup \dots \cup T^{\ell-1}(Z) = X.
\end{equation*}
We claim that $T^i (Z) \cap T^j(Z)$ has empty interior for all $i\neq j$ in $\{0,\dots,\ell-1\}$.
Suppose, on contrary, that $K \coloneqq T^i(Z) \cap T^j(Z)$ has non-empty interior for certain $i\neq j$ in $\{0,\dots,\ell-1\}$.
First, we note that $m \in F$ as $Y$ is $T^m$-invariant.
Hence, $\ell$ divides $m$.
This, the definition of $Z$ and the fact that $(Y,T^m)$ is transitive implies that $(Z,T^\ell)$ is transitive.
Now, $T^{-i}(K)$ is a closed $T^\ell$-invariant subset of $Z$ with non-empty interior (in $Z$).
Therefore, $T^{-i}(K) = Z$, and thus $Z \subseteq T^{j-i}(Z)$.
A symmetric argument shows that $Z \subseteq T^{i-j}(Z)$, so $Z = T^{i-j}(Z)$ and $|i-j| \in E$.
Now, since $i\neq j$, we have that $0 < |i-j| < \ell$, and thus that $|i-j|$ is an element of $E$ that is strictly smaller than $\ell$, which contradicts the definition of $\ell$.
This shows that $T^i (Z) \cap T^j(Z)$ has empty interior for all $i\neq j$ in $\{0,\dots,\ell-1\}$.
We conclude that $\ell \in \CAP(T)$.
The hypothesis then implies that $\gcd(m,\ell) = 1$, which is only possible, as $\ell$ divides $m$, when $\ell = 1$.
We obtain, using that $\ell \in E$, that $T(Y) \subseteq Y$, and thus that $Y = X$ as the $T$-transitive point $\tilde{x}$ lies in $Y$.
We have shown that $\hat{x}$ is transitive for $T^m$, and thus that $T^m$ acts transitively on $X$.
\end{proof}

% % % % % % % % % % LEMMA % % % % % % % % % %
\begin{lemma} \label{transitive_on_Xm_iff_coprimes}
Suppose that $(X,T)$ is a transitive system and $n \geq 1$.
Let $m \in \CAP(T)$, with $X_m$ defining a cyclic almost-partition of size $m$.
Then, $T^{nm}$ acts transitively on $X_m$ if and only if $\gcd(n,k/m) = 1$ for all $k \in \CAP(T)$ such that $m$ divides $k$.
\end{lemma}
\begin{proof}
\autoref{transpowers} ensures that $T^{nm}$ acts transitively on $X_m$ if and only if
\begin{equation}  \label{eq:transitive_on_Xm_iff_coprimes:1}
    \text{$\gcd(n,k) = 1$ for all $k \in \CAP(T^m|_{X_m})$.}
\end{equation}
Now, by \autoref{Peig:props}, $k \in \CAP(T^m|_{X_m})$ exactly when $mk \in \CAP(T)$.
Hence, \eqref{eq:transitive_on_Xm_iff_coprimes:1} is equivalent to
\begin{equation*}
    \text{$\gcd(n,k) = 1$ for all $mk \in \CAP(T)$.}
\end{equation*}
This is the same as $\gcd(n,k/m) = 1$ for all $k \in \CAP(T)$ such that $m$ divides $k$.
The lemma follows.
\end{proof}

% % % % % % % % % % LEMMA % % % % % % % % % %
\begin{proposition} \label{decompose:trans_peig}
Let $(X,T)$ be a transitive system and $n \geq 1$.
Then, there is a unique decomposition $n = k\cdot \ell$, where $k,\ell \geq 1$ are such that $T^k$ acts transitively on $X$ and $\ell \in \CAP(T)$.
\end{proposition}
\begin{proof}
Let $\ell \geq 1$ be the largest integer such that $\ell \in \CAP(T)$ and $\ell$ divides $n$.
We set $k = n / \ell$.
Suppose that $m \in \CAP(T)$.
Then, by \autoref{Peig:props}, we have that $\gcd(n,m) \in \CAP(T)$.
By the maximality of $\ell$, this implies that $\gcd(n,m)$ divides $\ell$.
Therefore, as $n = k \cdot \ell$, we have that $\gcd(k, m) = 1$.
As $m$ was chosen arbitrarily, \autoref{transpowers} ensures that $T^k$ acts transitively on $X$.

We are left with proving that the decomposition is unique.
Let us assume that $n = k'\ell'$ is another decomposition such that $T^{k'}$ acts transitively on $X$ and $\ell' \in \CAP(T)$.
Then, $n = k \ell = k' \ell'$.
Moreover, by \autoref{transpowers}, $\gcd(k,\ell') = 1$, so $k$ divides $k'$.
Similarly, $k'$ divides $k$. 
We conclude that $k = k'$, and thus that $\ell = \ell'$.
\end{proof}

\subsection{Description of the stabilized automorphism group}

% % % % % % % % % % LEMMA % % % % % % % % % %
The following lemma is central for our analysis.
We recall the reader that, by \autoref{eigenpartition}, a transitive system $(X,T)$ has a cyclic partition of size $m$ if and only if $m \in \Eig(T)$.
\begin{lemma} \label{propertiesofXk}
Suppose that $(X,T)$ is a transitive system.
Let $m \in \Eig(T)$, with $X_m \subseteq X$ defining a cyclic almost-partition of size $m$.
If $n \geq 1$ is such that $T^{nm}$ acts transitively on $X_m$ and $g \in \Aut(T^{nm})$, then there exists a permutation $\sigma \in \Sym(m)$ satisfying that $g(T^i X_m) = T^{\sigma(i)} X_n$ for all $i\in\{0,1,\dots,m-1\}$.
\end{lemma}
\begin{proof}
The definition of $X_m$ ensures that $X$ is equal to the disjoint union $X_m \cup \dots \cup T^{m-1}(X_m)$.
Hence, since $g$ is a homeomorphism, $g(T^k X_m) \cap T^{\sigma(k)}(X_m)$ has non-empty interior for some $\sigma(k) \in \{0,\dots,m-1\}$.
Moreover, as $g$ commutes with $T^{nm}$, the set $g(T^k X_m) \cap T^{\sigma(k)}(X_m)$ is $T^{nm}$-invariant.
Since the hypothesis ensures that $T^{nm}$ is transitive on each $T^{\sigma(k)}(X_m)$, we deduce that
\begin{equation} \label{eq:propertiesofXk:0}
    \text{$g(T^k X_m) \supseteq g(T^k X_m) \cap T^{\sigma(k)}(X_m) \supseteq T^{\sigma(k)}(X_m)$ for any $0 \leq k < m$.}
\end{equation}
Now, being $g$ a homeomorphism, $g(X_m) \cup g(T(X_m)) \cup \dots \cup g(T^{m-1}(X_m))$ is a partition of $X$.
In view of \eqref{eq:propertiesofXk:0}, this is possible only if $g(T^k X_m) = T^{\sigma(k)}(X_m)$.
\end{proof}

Let us recall the following notation: If $G$ and $H$ are groups, then 
\begin{equation*}
    C_G(H) = \{ g \in G : gh = hg,\ \forall h \in H\}.
\end{equation*}
Also, if $n \geq 1$ and $\sigma \in \Sym(n)$, then $g_\sigma = (g_{\sigma^{-1}(0)}, g_{\sigma^{-1}(1)}, \dots, g_{\sigma^{-1}(n-1)})$ for every $g = (g_0,\dots,g_{n-1}) \in G^n$.

% % % % % % % % % % PROPOSITION % % % % % % % % % %
\begin{proposition} \label{prop:C(gamma)_as_wreathProd}
Let $(X,T)$ be a transitive system, $\gamma \in \Aut(T)$, $m \in \Eig(T)$, and $n \geq 1$ be such that $T^n$ acts transitively on $X$.
For $r \in \Eig(T)$, let $X_r$ define a cyclic partition of size $r$.
Then, there is $j \in \{0,\dots,m-1\}$ such that $\gamma(X_m) = T^j X_m$ and, if $r = \gcd(j,m)$, then $r \in \Eig(T)$ and 
\begin{equation*}
    C_{\Aut(T^{mn})}(\gamma) \cong C_{\Aut(T^{mn}|_{X_r})}(\gamma|_{X_r}) \wr \Sym(r).
\end{equation*}
More precisely, the isomorphism is given by the split exact sequence
\begin{equation} \label{eq:prop:C(gamma)_as_wreath_product:exact_seq}
\begin{tikzcd}[column sep=3ex]
    1 \arrow{r} & C_{\Aut(T^{mn}|_{X_r})}(\gamma|_{X_r})^r 
    \arrow{r}{\psi} & C_{\Aut(T^{mn})}(\gamma)
    \arrow{r}{\pi} & \Sym(r) \arrow{r} 
    \arrow[bend left=33]{l}{\rho} & 1,
\end{tikzcd}
\end{equation}
which satisfies the following properties:
\begin{enumerate}
    \item If $g \in C_{\Aut(T^{mn}|_{X_r})}(\gamma|_{X_r})^r$, $g=(g_0,\dots,g_{r-1})$, $i \in \{0,\dots,r-1\}$ and $x\in X_r$ then $\psi(g)(T^i x) = T^i g_{i}(x)$.
    \item $g(T^i X_r) = T^{\pi(g)(i)} (X_r)$ for every $g \in C_{\Aut(T^{mn})}(\gamma)$ and $i \in \{0,\dots,r-1\}$.
    \item $\rho(\sigma)(T^i x) = T^{\sigma(i)}(x)$ for all $x \in X_r$ and $i \in \{0,\dots,r-1\}$.
    \item $\rho(\sigma)^{-1} \psi(g) \rho(\sigma) = \psi(g_\sigma)$ for all $\sigma \in \Sym(r)$ and $g \in C_{\Aut(T^{mn}|_{X_r})}(\gamma|_{X_r})^r$.
\end{enumerate}
\end{proposition}
\begin{proof}
We start with some general observations.
\autoref{eigenpartition} states that there is a set $X_m \subseteq X$ such that
\begin{equation} \label{eq:prop:C(gamma)_as_wreath_product:cyclic_part}
\text{$X_m$ is closed, $T^m$-invariant and $\bigcup_{k=0}^{m-1} T^k(X_m)$ is a partition of $X$.}
\end{equation}
Moreover, by \autoref{transitive_on_Xm_iff_coprimes}, $T^m$ acts transitively on each $T^k(X_m)$.
We use \autoref{propertiesofXk} to obtain a permutation $\eta  \in \Sym(m)$ such that  $\gamma(T^k(X_m)) = T^{\eta(k)}(X_m)$ for every $k \in \{0,\dots,m-1\}$.
Then, as $\gamma$ commutes with $T$, setting $j = \eta(0)$ yields
\begin{equation} \label{eq:prop:C(gamma)_as_wreath_product:gamma_permutes_nicely}
\gamma^k(X_m) = T^{kj}(X_m)
\ \text{for all $k \in \Z$.}
\end{equation}
Let $r = \gcd(j,m)$ and note that, since $r$ divides $m$, $r \in \Eig(T)$.
We set $X_r = \bigcup_{k\in\Z} T^{kj} (X_m)$.
Observe that Equations \eqref{eq:prop:C(gamma)_as_wreath_product:cyclic_part} and \eqref{eq:prop:C(gamma)_as_wreath_product:gamma_permutes_nicely} ensure that 
\begin{equation} \label{eq:prop:C(gamma)_as_wreath_product:defi_Xr}
    \text{$X_r$ is closed, $T^r$- and $\gamma$-invariant, and $\bigcup_{k=0}^{r-1} T^k(X_r)$ is a partition of $X$.}
\end{equation}
This permits to use \autoref{transitive_on_Xm_iff_coprimes} and deduce that 
\begin{equation} \label{eq:prop:C(gamma)_as_wreath_product:trans_action}
\text{$T^r$ acts transitively on each $T^k(X_r)$.}
\end{equation}

We now construct the exact sequence in \eqref{eq:prop:C(gamma)_as_wreath_product:exact_seq}.
\medskip

\textbf{Part 1} (Defining the map $\pi\colon C_{\Aut(T^{mn})}(\gamma) \to \Sym(r)$).
Let $g \in C_{\Aut(T^m)}(\gamma)$.
As $g$ is an automorphism of $T^m$, \autoref{propertiesofXk} gives $\sigma \in \Sym(m)$ such that $g(T^k X_m) = T^{\sigma(k)} (X_m)$ for all $k \in \{0,\dots,m-1\}$.
Then, as $g$ commutes with $\gamma$, we have that $g(\gamma(T^k X_m)) = \gamma(g(T^kX_m))$ for any $k \in \{0,\dots,m-1\}$.
We can use \eqref{eq:prop:C(gamma)_as_wreath_product:gamma_permutes_nicely} to manipulate each side of this equality as follows:
\begin{equation*}
g(\gamma(T^k X_m)) = 
g(T^{k+j} X_m) = T^{\sigma(k+j)} X_m
\end{equation*}
and 
\begin{equation*}
T^{\sigma(k) + j} X_m = 
\gamma(T^{\sigma(k)} X_m) =
\gamma(g(T^k X_m)).
\end{equation*}
We deduce that $\sigma(k+j) = \sigma(k) + j \pmod{m}$ for all $k \in \{0,\dots,m-1\}$.
This implies that 
\begin{equation*}
g(T^i X_r) = g(\bigcup_{k\in\Z} T^{i+kj} (X_m)) = 
\bigcup_{k\in\Z} T^{\sigma(i)+kj} (X_m) = T^{\sigma(i)} (X_r).
\end{equation*}
We conclude that $g$ permutes the elements of the set $\{ T^i (X_r) : i \in \{0,\dots,r-1\}\}$.
We define $\pi(g)$ as the element of $\Sym(r)$ satisfying $g(T^i X_r) = T^{\pi(g)(i)} (X_r)$ for all $i \in \{0,\dots,r-1\}$.
It not difficult to see that $\pi$ is a group-morphism.
\medskip

\textbf{Part 2} (Defining the map $\rho\colon \Sym(r) \to C_{\Aut(T^{mn})}(\gamma)$).
Let $\sigma \in \Sym(r)$.
Thanks to \eqref{eq:prop:C(gamma)_as_wreath_product:defi_Xr}, we can define $\rho(\sigma)\colon X \to X$ by $\rho(\sigma)(T^i x) = T^{\sigma(i)} (x)$ for $x \in X$ and $i \in \{0,\dots,r-1\}$.
It is clear that $\rho(\sigma) \in \Aut(T^m)$.
Moreover, since each $T^i(X_r)$ is $\gamma$-invariant and since $\gamma$ commutes with $T$, we have that $\rho(\sigma)$ commutes with $\gamma$.
Thus, $\rho$ maps $\Sym(r)$ into $C_{\Aut(T^m)}(\gamma)$.
It is then easy to see that $\rho$ is a group-morphism such that $\pi(\rho(\sigma)) = \sigma$ for all $\sigma \in \Sym(r)$.
In particular, $\pi$ is onto.
That $\rho$ is injective is a direct consequence of its definition.
\medskip

\textbf{Part 3} (Defining the map $\psi\colon C_{\Aut(T^{mn}|_{X_r})}(\gamma|_{X_r})^r \to C_{\Aut(T^{mn})}(\gamma)$).
Each of the sets $T^i (X_r)$ is $\gamma$-invariant and $T^m$-invariant (the latter being true because $r$ divides $m$), so $\tilde{\gamma} \coloneqq \gamma|_{X_r}$ and $\tilde{T} \coloneqq T^m$ are homeomorphisms of $X_r$.
We define $G = C_{\Aut(\tilde{T})}(\tilde{\gamma})$ and, for $g = (g_0,\dots,g_{r-1}) \in G^r$, the map $\psi(g)$ by $\psi(g)(T^i x) = T^i g_i(x)$ for $x \in X_r$ and $i \in \{0,\dots,r-1\}$.
By \eqref{eq:prop:C(gamma)_as_wreath_product:defi_Xr}, $\psi(g)$ is well-defined and a homeomorphism of $X$.
Note that $\psi(g)$ maps $T^i (X_r)$ onto $T^i (X_r)$ for all $i \in \{0,\dots,r-1\}$; in particular, since every $g_i$ commutes with $\tilde{T}$ and $\tilde{\gamma}$, it can be checked that $\psi(g)$ commutes with $T^m$ and $\gamma$.
Therefore, $\psi$ maps $G$ into $C_{\Aut(T^m)}(\gamma)$.
A routine computation shows that $\psi$ is an injective group-morphism.

Let us check that $\psi(G^r) = \ker \pi$.
Any $g \in \psi(G)$ fixes each $T^i (X_r)$, so $\pi(g) = 1$, that is, $g \in \ker \pi$.
Conversely, if $g \in \ker \pi$, then the maps $g_i(x) = T^{-i} g (T^i x)$, which are defined for $x \in X_r$ and $i \in \{0,\dots,r-1\}$, are homeomorphisms of $X_r$.
They also commute with $\tilde{T}$ and $\tilde{\gamma}$ as $g$ commutes with $T^m$ and $\gamma$.
Hence, $(g_0,\dots,g_{r-1}) \in G^r$.
We can then note that the definition of $\psi$ ensures that $g = \psi(g_0,\dots,g_{r-1})$, which shows that $g \in \psi(G^r)$.
This proves that $\psi(G^r) = \ker \pi$.
\medskip

We have proved that the sequence \eqref{eq:prop:C(gamma)_as_wreath_product:exact_seq} is exact, and thus that $C_{\Aut(T^m)}(\gamma)$ is isomorphic to a semi-direct product $G^r \rtimes \Sym(r)$.
Item (1) of the proposition then follows from the definitions of $\pi$ and $\psi$.

\textbf{Part 4} (The semi-direct product arising from \eqref{eq:prop:C(gamma)_as_wreath_product:exact_seq} is a wreath product).
We are left with proving Item (2) and that the action of $\Sym(r)$ on $G^r$ coincides with that of the wreath product.
Let $\sigma \in \Sym(r)$ and $g = (g_0,\dots,g_{r-1}) \in G^r$.
We write $g_\sigma = (g_{\sigma(0)}, g_{\sigma(1)}, \dots, g_{\sigma(r-1)})$.
Then, for $x \in X_r$ and $i \in \{0,\dots,r-1\}$, we have that 
\begin{multline*}
\pi(\sigma)^{-1} \psi(g) \pi(\sigma)( T^i x ) =
\pi(\sigma)^{-1} \psi(g) ( T^{\sigma(i)} x ) \\ =
\pi(\sigma)^{-1} ( T^{\sigma(i)} g_{\sigma(i)}(x) ) =
g_{\sigma(i)}(x) = \pi(g_\sigma)(T^i x).
\end{multline*}
This gives Item (2). 
It follows that the semi-direct product arising from \eqref{eq:prop:C(gamma)_as_wreath_product:exact_seq} is the wreath product $G \wr \Sym(r)$.
\end{proof}

\begin{theorem} \label{semidirectproduct}
Let $(X,T)$ be a transitive system, $m \in \Eig(T)$ and $n \geq 1$ be such that $T^n$ acts transitively on $X$.
Let $X_m$ define a cyclic partition of size $m$.
Then:
\begin{equation} \label{eq:semidirectproduct:iso}
    \Aut(T^{nm}) \cong \Aut(T^{nm}|_{X_m}) \wr \Sym(m).
\end{equation}
Moreover, we have the following isomorphisms:
\begin{enumerate}[(i)]
    \item $ \Aut(T)/\langle T^m\rangle\cong  \Aut(T^m|_{X_m})/\langle T^m|_{X_m}\rangle\times \Z/m\Z$.
    \item $ \Aut(T)/\langle T\rangle\cong  \Aut(T^m|_{X_m})/\langle T^m|_{X_m}\rangle$.
\end{enumerate}
\end{theorem}
\begin{proof}
We set $\gamma \coloneqq T^{nm} \in \Aut(T)$.
Note that $j \coloneqq 0$ satisfies $\gamma(X_m) = T^j(X_m)$ and that $r \coloneqq \gcd(j,m) = m$.
Therefore, \autoref{prop:C(gamma)_as_wreathProd} ensures that
\begin{multline}  \label{eq:prop:from_eigPart_to_wreathProd:base_iso}
    \Aut(T^{nm}) = 
    C_{\Aut(T^{nm})}(\gamma) \\ \cong 
    C_{\Aut(T^{nm}|_{X_r})}(\gamma|_{X_r}) \wr \Sym(r) = 
    \Aut(T^{nm}|_{X_m}) \wr \Sym(m).
\end{multline}

Let us now prove Item (i) and (ii).
We set $n = 1$ and denote $G = \Aut(T^m|_{X_m})$.
\autoref{prop:C(gamma)_as_wreathProd} describes the wreath product in \eqref{eq:prop:from_eigPart_to_wreathProd:base_iso} as an exact sequence consisting of the maps $\psi\colon G^m \to \Aut(T^m)$, $\pi\colon \Aut(T^m) \to \Sym(m)$ and $\rho\colon \Sym(m) \to \Aut(T^m)$.
Let $\phi \colon \Aut(T^n) \to G \wr \Sym(m)$ be the isomorphism that appears in \eqref{eq:prop:from_eigPart_to_wreathProd:base_iso}.
Before proving Items (i) and (ii), we describe $\phi(\Aut(T))$.

Let $g \in \Aut(T)$ and $\ell = \pi(g)(0)$.
Note that \eqref{eq:prop:from_eigPart_to_wreathProd:base_iso} states that $g(X_m) = T^\ell(X_m)$.
Hence, as $g$ commutes with $T$, $g(T^k X_m) = T^{k+\ell}(X_m)$ for all $k \in \{0,\dots,m-1\}$.
This implies that if $c \in \Sym(m)$ is the cyclic permutation $c(k) = k+1 \pmod{m}$, then $\pi(g) = c^\ell$.

Let $g' = \rho(c)^{-\ell} g$.
Being $\pi$ a left-inverse of $\rho$, we have that $g' \in \ker \pi$.
Hence, as the sequence \eqref{eq:prop:from_eigPart_to_wreathProd:base_iso} is exact, we can define $\vec{g} = (g_0,\dots,g_{m-1}) = \psi^{-1}(g')$. 
We then have $\phi(g') = (\vec{g}, 1)$, and thus that
\begin{equation} \label{eq:prop:from_eigPart_to_wreathProd:phi(g)_as_product}
    \phi(g) = \phi(\rho(c^\ell)) \phi(g') =
    (1, c^\ell) (\vec{g}, 1) = (\vec{g}, c^\ell).
\end{equation}
We now describe $\vec{g}$.
Recall that $\rho(c^\ell)(T^k x) = T^{c^\ell(k)} x$ for any $k \in \{0,\dots,m-1\}$ and $x \in X_m$.
Hence, by Item (1) of \autoref{prop:C(gamma)_as_wreathProd},
\begin{equation} \label{eq:prop:from_eigPart_to_wreathProd:tmp1}
    T^k g_k(x) = g'(T^k x) =
    \rho(c^\ell)^{-1}(g(T^k x)) =
    T^{-c^\ell(k)} g(T^k x),
\end{equation}
for all $k \in \{0,\dots,m-1\}$ and $x \in X_m$.
Now, as $g$ commutes with $T$, we can write
\begin{equation*}
    T^{-c^\ell(k)} g(T^k x) = 
    T^{-c^\ell(k) + k} g(x) = 
    T^{-c^\ell(k) + k + c^\ell(0)} g_0(x),
\end{equation*}
where in the last step we used \eqref{eq:prop:from_eigPart_to_wreathProd:tmp1}.
We deduce, by \eqref{eq:prop:from_eigPart_to_wreathProd:tmp1}, that $g_k = T^{-c^\ell(k) + c^\ell(0)} g_0$ for all $k \in \{0,\dots,m-1\}$.
Since $c^\ell(k) = k + \ell$ if $0 \leq k < m-\ell$ and $c^\ell(k) = k + \ell - m$ if $m-\ell \leq k < m$, we have that $g_k = g_0$ for $0 \leq k < m-\ell$ and $g_k = T^m g_0$ if $m-\ell \leq k < m$.
Therefore, by \eqref{eq:prop:from_eigPart_to_wreathProd:phi(g)_as_product}, the following  holds for any $g \in \Aut(T)$:
\begin{equation} \label{eq:prop:from_eigPart_to_wreathProd:description_phi_Aut(T)}
    \phi(g) = 
    ((\underset{\text{$m-\ell$ times}}{\underbrace{g_0,\dots,g_0}}, 
      \underset{\text{$\ell$ times}}{\underbrace{T^mg_0,\dots, T^m g_0}}), c^\ell),
\end{equation}
with $\ell = \pi(g)(0)$ and $g_0$ the first coordinate of $\psi^{-1}(\rho(c)^{-\ell}g)$.
With this description of $\phi(\Aut(T))$, we prove Items (2) and (3).

We define $S = T^m|_{X_m}$ and
\begin{equation*}
    \text{$\theta \colon \Aut(T)/\langle T^m\rangle \to G / \langle S\rangle \times \Z/m\Z$ by 
            $\theta(g\langle T^m\rangle ) = (g_0 \langle S\rangle, \ell)$,}
\end{equation*}
where $\ell = \pi(g)(0)$ and $g_0$ is the first coordinate of $\psi^{-1}(\rho(c)^{-\ell}g)$.
It follows from \eqref{eq:prop:from_eigPart_to_wreathProd:description_phi_Aut(T)} and standard computations that $\theta$ is well-defined and an isomorphism of groups.
In particular, Item (i) holds.
Observe that \eqref{eq:prop:from_eigPart_to_wreathProd:description_phi_Aut(T)} also implies that $\theta$ sends $\langle T \rangle / \langle T^m\rangle$ onto the subgroup $\{1\} \times \Z/m\Z$ of $G / \langle S\rangle \times \Z/m\Z$.
Therefore,
\begin{multline*}
    \Aut(T) / \langle T \rangle \cong
    (\Aut(T)/\langle T^m\rangle) / (\langle T \rangle / \langle T^m\rangle) \\ \cong
    (G / \langle S\rangle \times \Z/m\Z) / (\{1\} \times \Z/m\Z) \cong
    G / \langle S \rangle.
\end{multline*}
This proves Item (ii) and completes the proof.
\end{proof}

We now turn to proving the final theorem of this section.
It shows that, if $\Eig(T) = \CAP(T)$ holds for a transitive system $(X,T)$, then $\StabAut(T)$ decomposes in such a way that the roles of the transitive powers of $T$ and of the rational eigenvalues of $T$ are independent.
This plays a key role in Section \ref{sec:recover_eigs} for retrieving the rational spectrum of $(X,T)$ from its stabilized automorphism group.

If $(X,T)$ is a transitive system, then we denote by $M(T)$ the set of all $n \geq 1$ for which $T^n$ acts transitively on $X$.
We further define, for $F \subseteq \N$,
\begin{equation*}
    \Aut^F(T) = \bigcup_{n \in F} \Aut(T^n).
\end{equation*}
Note that $\Aut^F(T)$ is a group whenever $\lcm(n,n') \in F$ for all $n,n' \in F$.
In particular, by \autoref{transpowers},
\begin{equation*}
    \text{$\Aut^{M(T)}(S)$ is a group for all transitive systems $(X,T)$ and $(Y,S)$.}
\end{equation*}

The following lemma is used only in Section \ref{sec:recover_eigs}.

\begin{lemma} \label{M(T^q)}
Suppose that $(X,T)$ is a transitive system.
Let $q \in \Eig(T)$ and $X_q$ be the clopen set given by \autoref{eigenpartition}.
We denote $S = T^q|_{X_q}$.
Then, $M(T) \subseteq M(S)$ and $\Aut^{M(T)}(S)$ is a subgroup of $\Aut^{M(S)}(S)$.
\end{lemma}
\begin{proof}
Let $n \in M(T)$.
Then, by \autoref{transpowers}, $\gcd(n,k) = 1$ for all $k \in \CAP(T)$.
In particular, $\gcd(n,k/q) = 1$ for all $k \in \CAP(T)$ such that $q$ divides $k$.
So, by \autoref{transitive_on_Xm_iff_coprimes}, $n \in M(S)$.
The second part of the lemma follows from the first one.
\end{proof}

We now prove a more general version of \autoref{semidirectproduct}.

% % % % % % % % % % THEOREM % % % % % % % % % %
\begin{theorem} \label{aut_as_eig&min_parts}
Let $(X,T)$ be a transitive system such that $\Eig(T) = \CAP(T)$.
For $m \in \Eig(T)$, let $X_m$ define a cyclic partition of size $m$.
Then, $\StabAut(T)$ is isomorphic to a direct limit
\begin{equation*}
    \varinjlim_{m \in \Eig(T)}
    \Aut^{M(T)}(T^m|_{X_m}) \wr \Sym(m).
\end{equation*}
\end{theorem}
\begin{proof}
Let us denote $G_{m,n} = \Aut(T^{nm}|_{X_m})$ for $m \in \Eig(T)$ and $n \in M(T)$.
We start by proving that
\begin{equation} \label{eq:aut_as_eig&min_parts:big_union}
    \StabAut(T) = \bigcup_{m \in \Eig(T)} \Aut^{M(T)}(T^m).
\end{equation}
Let $k \geq 1$ be arbitrary.
We use \autoref{decompose:trans_peig} to write $k = nm$, with $n \in M(T)$ and $m \in \CAP(T)$.
Then, as $n \in M(T)$, \autoref{transpowers} ensures that $\gcd(n,m/k) = 1$ for all $k \in \CAP(T)$, so $n \in M(T^m)$ by \autoref{transitive_on_Xm_iff_coprimes}.
Hence, $\Aut(T^k) \subseteq \Aut^{M(T)}(T^m)$.
Being $m \in \Eig(T)$ since $m \in \CAP(T)$ and $\Eig(T) = \CAP(T)$, $\Aut(T^k)$ is contained in the set in the right-hand side of \eqref{eq:aut_as_eig&min_parts:big_union}.
As $k$ was chosen arbitrarily, we deduce that \eqref{eq:aut_as_eig&min_parts:big_union} holds.

Next, we use \autoref{semidirectproduct} with $n \in M(T)$ and $m \in \Eig(T)$ to obtain an isomorphism
\begin{equation} \label{eq:aut_as_eig&min_parts:level_iso}
    \phi_{m,n} \colon \Aut(T^{nm}) \to G_{m,n} \wr \Sym(m).
\end{equation}
Note that if $n,n' \in M(T)$ and $n$ divides $n'$, then $G_{m,n}$ is contained in $G_{m,n'}$.
Thus, by the description of the wreath appearing in \eqref{eq:aut_as_eig&min_parts:level_iso} from \autoref{semidirectproduct}, we have that the restriction of $\phi_{m,n'}$ to $\Aut(T^{nm})$ is equal to $\phi_{m,n}$.
Therefore, there is an induced isomorphism
\begin{equation*}
    \phi_m \colon \Aut^{M(T)}(T^m) \to (\bigcup_{n\in M(T)} G_{m,n}) \wr \Sym(m).
\end{equation*}
Note that
\begin{equation*}
    G_m \coloneqq \Aut^{M(T)}(T^m|_{X_m}) = \bigcup_{n\in M(T)} G_{m,n},
\end{equation*}
so $\phi_m$ maps $\Aut^{M(T)}(T^m)$ onto $G_m \wr \Sym(m)$. 
We conclude, using \eqref{eq:aut_as_eig&min_parts:big_union}, that 
\begin{equation*}
    \StabAut(T) = 
    \bigcup_{m \in \Eig(T)} \phi_n^{-1}(G_m \wr \Sym(m)) 
    \cong 
    \varinjlim_{m \in \Eig(T)}
    G_m \wr \Sym(m).
\end{equation*}
\end{proof}

\begin{corollary}
% Let $(X,T)$ be a transitive system such that $\Eig(T) = \CAP(T)$. 
% Let $X_m$ be as in the previous theorem. 
Assume the notation and hypothesis of \autoref{aut_as_eig&min_parts}.
If \\ $\Aut^{M(T)}(T^m|_{X_m})$ is amenable for all $m\in \Eig(T)$, then $\StabAut(T)$ is amenable.
\end{corollary}
\begin{proof}
This is a consequence of two well known properties of amenable groups. 
First, if a group $B$ is an extension of $C$ by $A$ and the groups $A$ and $C$ are amenable, then $B$ is amenable.
This property implies that $\Aut^{M(T)}(T^m|_{X_m}) \wr \Sym(m)$ is amenable for all $m\in \Eig(T)$. 
The second property is that amenability is preserved under direct limits. 
\end{proof}

% % % % % % % % % % SECTIONS % % % % % % % % % %
% % % % % % % % % % SECTION SECTION % % % % % % % % % % 
% % % % % % % % % % SECTION SECTION % % % % % % % % % % 
\section{Key algebraic properties of wreath products} \label{section:algebra}

The objective of this section is to prove two technical algebraic results: Theorem \ref{prop:main_alg_lemma} and Proposition \ref{prop:main_alg_lemma2}.
They state that a wreath-product-like structure of a group is preserved, to some extent, by group isomorphisms.
These two results are crucial for recovering the rational eigenvalues of a system from its stabilized automorphism group in Section \ref{sec:recover_eigs}.

Let us start by introducing the necessary notation.
We fix, for the rest of the subsection, a group $G$ and $n \geq 2$.
Recall that $\Sym(n)$ acts on $G^n$ by permuting coordinates.
More precisely, if $g = (g_0,\dots,g_{n-1})$ and $\sigma \in \Sym(n)$, then we consider the left-action $g_\sigma = (g_{\sigma^{-1}(0)}, \dots, g_{\sigma^{-1}(n-1)})$.
The wreath product $G \wr \Sym(n)$ is the semi-direct product $G^n \rtimes \Sym(n)$ with operation defined by
\begin{equation} \label{defi_wreath_product}
    (g,\sigma) (h, \tau) = (g_{\tau^{-1}} h, \sigma \tau) \ 
    \text{for $g,h \in G^n$ and $\sigma, \tau \in \Sym(n)$.}
\end{equation}
In $G \wr \Sym(n)$, we define
\begin{equation*}
\text{$(h,\tau)^{(g,\sigma)} \coloneqq (g,\sigma) (h,\tau) (g,\sigma)^{-1}$ and $[(g,\sigma), (h,\tau)] \coloneqq (g,\sigma) (h,\tau) (g,\sigma)^{-1} (h,\tau)^{-1}$. }
\end{equation*}
If $H$ is a subgroup of $G \wr \Sym(n)$, we say that $(g,\sigma), (h,\tau) \in G \wr \Sym(n)$ are {\em conjugate in $H$} if there is $(k,\eta) \in H$ such that $(g,\sigma)^{(k,\eta)} = (h,\tau)$.
When $H = G \wr \Sym(n)$, we simply say that $(g,\sigma), (h,\tau) \in G \wr \Sym(n)$ are {\em conjugate}.

The following lemma summarizes the straightforward computations of inverses, conjugates, and commutators of elements of $G \wr \Sym(n)$.
% % % % % % % % % % LEMMA LEMMA % % % % % % % % % % 
\begin{lemma} \label{lem:wreath_prod_formulas}
Let $(g,\sigma), (h,\tau) \in G \wr \Sym(n)$.
Then:
\begin{enumerate}[(i)]
    \item \label{lem:wreath_prod_formulas:inverse}
    $(g,\sigma)^{-1} = (g^{-1}_{\sigma}, \sigma^{-1})$.
    \item \label{lem:wreath_prod_formulas:conjugation}
    $(h,\tau)^{(g,\sigma)} = (g_{\tau^{-1}\sigma} h_\sigma g^{-1}_\sigma, \sigma\tau\sigma^{-1})$.
    \item \label{lem:wreath_prod_formulas:commutator}
    $ [(g,\sigma), (h,\tau)] = (g_{\tau^{-1}\sigma\tau} h_{\sigma\tau} g^{-1}_{\sigma\tau} h^{-1}_\tau, \sigma\tau\sigma^{-1}\tau^{-1}). $
\end{enumerate}
\end{lemma}
%\begin{proof}
%The formulas are a consequence of \autoref{defi_wreath_product}.
%\end{proof}

One of the central properties that we use for proving Theorem \ref{prop:main_alg_lemma} and Proposition \ref{prop:main_alg_lemma2} is that symmetric groups have a very rigid normal subgroup structure.
% We also need to know the structure of normal subgroups of the symmetric groups.
This structure is summarized in \autoref{lem:symmetric_groups_structure} below.
For $m \geq 1$, we denote by $A_m$ the alternating group on $m$ elements.

% % % % % % % % % % LEMMA LEMMA % % % % % % % % % % 
\begin{lemma} \label{lem:symmetric_groups_structure}
Let $m \geq 1$.
The normal subgroups of $\Sym(m)$ are $\{1\}$, $A_m$, $\Sym(m)$, and, if $m = 4$, we also have the normal subgroup $V$ of $A_4$ that is isomorphic to the Klein four-group.
The normal subgroups of $A_m$ are $\{1\}$, $A_m$ and, if $m = 4$, $V$.
\end{lemma}

The following two lemmas detail the consequences of the rigidity expressed in Lemma \ref{lem:symmetric_groups_structure}.

% % % % % % % % % % LEMMA LEMMA % % % % % % % % % % 
\begin{lemma} \label{lem:subgroups_with_=_card_sym_groups}
Let $m \geq 2$ and $L, L' \trianglelefteq K \trianglelefteq \Sym(m)$.
Assume that $|L| = |L'|$ and that $K$ is not the Klein group $V$.
Then, $L = L'$.
\end{lemma}
\begin{proof}
If $K = \{1\}$, then $L = L' = \{1\}$ and the proof is complete.
Let us assume that $K \not= \{1\}$.
Then, as $K \not= V$, \autoref{lem:symmetric_groups_structure} ensures that $K \supseteq A_m$.
This implies, by \autoref{lem:symmetric_groups_structure}, that the only normal subgroups of $K$ can be $\{1\}$, $V$, $A_m$ and $\Sym(m)$.
Observe that any pair of these groups either are equal or have different cardinalities.
Since $|L| = |L'|$, we deduce that $L = L'$.
\end{proof}

\begin{lemma} \label{lem:LisTransIfHasComplement}
Let $L \leq \Sym(n)$ and $K \trianglelefteq \Sym(n)$ be such that
\begin{equation*}
    L\cdot K \coloneqq \{gh : g \in L, h \in K\} = \Sym(n).
\end{equation*}
Then, $K \supseteq A_n$ or $L$ acts transitively on $\{0,1,\dots,n-1\}$.
\end{lemma}
\begin{proof}
Suppose that $K$ does not contain $A_n$ and let us prove that $L$ acts transitively on $\{0,1,\dots,n-1\}$.
We have, by \autoref{lem:symmetric_groups_structure}, that either $K = \{1\}$ or $n = 4$ and $K$ is equal to the normal subgroup of $A_4$ that is isomorphic to the Klein group $V$.

In the first case, we have that $L = \Sym(n)$ and thus the lemma is true.
Let us then consider the second case.
We use cycle notation for permutations, that is, if $a_1,\dots,a_k$ are different elements of $\{0,\dots,n-1\}$, then $(a_1\ a_2\ \dots\ a_k)$ denotes the permutation $\sigma$, where $\sigma(a_j) = a_{j+1}$ for $1 \leq j < k$, $\sigma(a_k) = a_1$, and $\sigma(a) = a$ for $a \not\in \{a_1,\dots,a_k\}$.
As we are in the second case, we can write:
\begin{equation*}
    K = V = \{1, (0\ 1)(2\ 3), (0\ 2)(1\ 3), (0\ 3)(1\ 2)\}.
\end{equation*}
Hence, the following sets are cosets of $K$ in $\Sym(4)$:
\begin{equation*}
    \{(0\ 1\ 2), (0\ 2\ 3), (1\ 3\ 2), (0\ 3\ 1)\}
    \ \text{and} \
    \{(0\ 2\ 1), (0\ 3\ 2), (1\ 2\ 3), (0\ 1\ 3)\}.
\end{equation*}
Now, since $L \cdot K = \Sym(4)$, $L$ contains at least one element of each coset of $K$.
It follows that $L$ contains two different cycles of length $3$.
This implies that $L$ acts transitively on $\{0,1,2,3\}$.
\end{proof}

The proof of the main results of this section is based on examining certain normal subgroups of $G \wr \Sym(n)$ and prove that they have an specific structure.
To achieve this, our main tool is \autoref{lem:fix_by_null}, which gives a method for obtaining new elements in a given normal subgroup $H$ of $G \wr \Sym(n)$ from other elements of $H$.

It is convenient to first introduce the following notation.
If $\sigma \in \Sym(n)$, then the {\em $\sigma$-orbit} of $i \in \{0,1,\dots,n-1\}$ is the set $p = \{\sigma^k(i) : k \in \Z\}$.
Remark that any $\sigma$-orbit $p$ has finite cardinality $|p|$ and that $\sigma^{-|p|}(j) = j$ for all $j \in p$.
If $j \in p$ and $g \in G^n$, then we define $c_{\sigma}(g,j) = g_{\sigma^{-|p|+1}(j)} g_{\sigma^{-|p|+2}(j)} \cdots g_j$.

Before proving \autoref{lem:fix_by_null}, we analyze the combinatorics of conjugating by elements of $G^n \times \{1\}$.

% % % % % % % % % % LEMMA LEMMA % % % % % % % % % % 
\begin{lemma}\label{lem:fix_by_c}
Let $(g,\sigma), (h,\sigma) \in G \wr \Sym(n)$ and, for each $\sigma$-orbit $p$, let $j_p \in p$.
Assume that
\begin{equation}
\text{$c_{\sigma}(g,j_p) = c_{\sigma}(h,j_p)$ for all $\sigma$-orbit $p$.}
\end{equation}
Then, $(h,\sigma)$ is conjugate to $(g,\sigma)$ in $G^n \times \{1\}$.
\end{lemma}
\begin{proof}
We define an element $k = (k_0,\dots,k_{n-1}) \in G^n$ as follows.
Let $p$ be a $\sigma$-orbit.
We set $k_{j_p} = 1$ and, if $|p| \geq 2$, we inductively define 
\begin{equation} \label{eq:lem:fix_by_c:0}
    \text{$k_{\sigma^{-l-1}(j_p)} = h_{\sigma^{-l}(j_p)} k_{\sigma^{-l}(j_p)} g_{\sigma^{-l}(j_p)}^{-1}$ for $l \in \{0,\dots,|p|-2\}$.}
\end{equation}
Since every $i \in \{1,\dots,n\}$ belongs to a $\sigma$-orbit, $k$ is completely defined.
Note that, by \eqref{eq:lem:fix_by_c:0},
\begin{equation} \label{eq:lem:fix_by_c:1}
    k_{\sigma^{-l-1}(j_p)} g_{\sigma^{-l}(j_p)} k^{-1}_{\sigma^{-l}(j_p)} = h_{\sigma^{-l}(j_p)}
\end{equation}
for all $\sigma$-orbit $p$ and $l \in \{0,\dots,|p|-2\}$.
An inductive use of this relation yields
$$ k_{\sigma^{|p|-1}(j_p)} (g_{\sigma^{|p|-2}(j_p)} g_{\sigma^{|p|-3}(j_p)}\dots g_{j_p}) k_{j_p} = 
h_{\sigma^{|p|-2}(j_p)} h_{\sigma^{|p|-3}(j_p)}\dots h_{j_p}. $$
Hence, as $c_{\sigma}(g,j_p) = c_{\sigma}(h,j_p)$ and $k_{j_p} = 1$, $k_{\sigma^{|p|-1}(j_p)} g_{\sigma^{|p|-1}(j_p)}^{-1} = h_{\sigma^{|p|-1}(j_p)}^{-1}$.
This equation can be written as $k_{j_p} g_{\sigma^{|p|-1}(j_p)} k_{\sigma^{|p|-1}(j_p)}^{-1} = h_{\sigma^{|p|-1}(j_p)}$.
The last relation is valid for all $p$, so, together with \eqref{eq:lem:fix_by_c:1}, it gives that $k_{\sigma^{-1}} g k^{-1} = h$.
We conclude, using Item \ref{lem:wreath_prod_formulas:conjugation} of Lemma \ref{lem:wreath_prod_formulas}, that 
\begin{equation*}
    (h,\sigma) = (k_{\sigma^{-1}} g k^{-1}, \sigma) =
    (g,\sigma)^{(k,1)}.
\end{equation*}
\end{proof}

% % % % % % % % % % LEMMA LEMMA % % % % % % % % % % 
\begin{lemma} \label{lem:fix_by_null}
Let $H \leq G \wr \Sym(n)$, $(h,\tau) \in H$, and take $j_p \in p$ for each $\tau$-orbit $p$.
Assume that $(k,1) (h,\tau) (k,1)^{-1} \in H$ for all $k \in G^n$.
If $g = (g_0,\dots, g_{n-1}) \in G^n$ is such that $c_{\tau}(g,j_p) = 1$ for all $\tau$-orbit $p$, then $(g,1) \in H$.
\end{lemma}
\begin{proof}
We have, by \autoref{lem:fix_by_c}, that there is $k = (k_0,\dots,k_{n-1}) \in G^n$ such that $k_{\sigma^{-|p|+1}(j_p)} = c_{\sigma}(g,j_p)$ and $k_{\sigma^{-l}(j_p)} = 1$ for all $\tau$-orbit $p$ and $l \in \{0,\dots,|p|-2\}$, and for which $(k,\tau)$ is conjugate to $(h,\tau)$ by an element of $G^n \times \{1\}$.
Let $g = (g_0,\dots,g_{n-1}) \in G^n$ be any element satisfying $c_{\tau}(g,j_p) = 1$ for all $\tau$-orbit $p$.
This condition guarantees that $c_\tau(k g,j_p) = c_\tau(g,j_p)$ for every $\tau$-orbit $p$.
Hence, we can use \autoref{lem:fix_by_c} and the hypothesis on $H$ to conclude that % inglés
\begin{equation*}
    \text{$(k g, \tau) \in H$ for any $g \in G^n$ such that $c_{\tau}(g,j_p) = 1$ for all $\tau$-orbit $p$.}
\end{equation*}
Now, being $(k,\tau)$ conjugate to $(h,\tau)$ in $G^n \times \{1\}$, we have that $(k,\tau) \in H$.
Hence, $(k,\tau)^{-1} (k g, \tau) \in H$.
We obtain that
\begin{equation*}
    H \ni (k, \tau)^{-1} (kg, \tau) = 
    (k^{-1}_\tau, \tau^{-1}) (kg, \tau) = 
    (k^{-1} kg, 1) = (g, 1).
\end{equation*}
\end{proof}

In the next result, we compute the centralizers of certain subgroups of $G \wr \Sym(n)$.
This and the fact that centralizers ar preserved under group isomorphisms are used in the proof of \autoref{lem:GisoHifTransAction} for transferring information through isomorphisms.

\autoref{lem:computing_centralizers} uses the following terminology.
We denote by $C(H)$ the centralizer of a subgroup $H$.
If $H$ and $K$ are subgroups of $G \wr \Sym(n)$, then $H\cdot K \coloneqq \{hk : h \in H, k \in K\}$.
If $H \leq G$, then $\Delta^n_H = \{(h,h,\dots,h) \in G^n : h \in H\}$.

\begin{lemma} \label{lem:computing_centralizers}
Assume that $n \geq 3$ and that $C(G) \neq \{1\}$.
Let $\pi\colon G \wr \Sym(n) \to \Sym(n)$ be the factor onto the last coordinate.
Then:
\begin{enumerate}
    \item $C(G \wr \Sym(n)) = \Delta^n_{C(G)} \times \{1\}$. 
    \item If $H \leq G \wr \Sym(n)$ is such that $\pi(H) = \Sym(n)$, then $C((\Delta^n_{C(G)} \times \{1\}) \cdot H)$ is contained in $\Delta^n_G \times \{1\}$.
    \item $C(\Delta^n_{C(G)} \times \Sym(n)) = \Delta^n_G \times \{1\}$.
\end{enumerate}
\end{lemma}
\begin{proof}
We start with some observations.
The condition $n \geq 3$ ensures that 
\begin{multline} \label{eq:lem:computing_centralizers:1}
    \text{$[\Sym(n), \Sym(n)] \coloneqq 
    \{ghg^{-1}h^{-1} : g,h \in \Sym(n)\}$} \\ 
    \text{acts transitively on $\{0,1,\dots,n-1\}$.}
\end{multline}
Using Item \ref{lem:wreath_prod_formulas:conjugation} of \autoref{lem:wreath_prod_formulas}, we get that, for any $(g,\sigma)$ and $(h,\tau)$ in $G \wr \Sym(n)$,
\begin{equation} \label{eq:lem:computing_centralizers:2}
    \text{$(g,\sigma)$ commutes with $(h,\tau)$ if and only if $g_{\tau^{-1}} h = h_{\sigma^{-1}} g$ and $\tau \sigma = \sigma \tau$.}
\end{equation}

We now prove Item (1).
If follows from \eqref{eq:lem:computing_centralizers:2} that $\Delta^n_{C(G)} \times \{1\}$ is a subset of $C(G \wr \Sym(n))$.
Let $(g,\sigma)$ be an element of $C(G \wr \Sym(n))$.
Assume, with the aim of obtaining a contradiction, that $\sigma \not= 1$.
Note that there is $i \in \{0,1,\dots,n-1\}$ such that $\sigma^{-1}(i) \neq i$.
Hence, we can find (using that $G \neq \{1\}$) an element $h \in G^n$ satisfying $h_i = 1$ and $h_{\sigma^{-1}(i)} \neq 1$.
Then, as $(g,\sigma)$ has to commute with $(h,1)$, \eqref{eq:lem:computing_centralizers:2} ensures that
\begin{equation*}
    g_i = g_i h_i = h_{\sigma^{-1}(i)} g_i.
\end{equation*}
In particular, $h_{\sigma^{-1}(i)} = 1$, which is a contradiction.
Therefore, $\sigma = 1$.
Then, by \eqref{eq:lem:computing_centralizers:2}, $gh = hg$ for all $h \in G^n$, so $g \in C(G)^n$.
This and \eqref{eq:lem:computing_centralizers:2} yield that $g_\tau = g$ for all $\tau \in \Sym(n)$.
We conclude that $g \in \Delta^n_G$, and thus that $(g,\sigma) \in \Delta^n_{C(G)} \times \{1\}$.
This proves Item (1).

We continue with Item (2).
Let $(g,\sigma) \in C(\Delta^n_{C(G)} \cdot H)$ be arbitrary.
By \eqref{eq:lem:computing_centralizers:2}, $\sigma$ commutes with $\tau$ for all $(h,\tau)$ in $(\Delta^n_{C(G)} \times \{1\}) \cdot H$.
Since $\pi(H) = \Sym(n)$, this implies that $\sigma \in C(\Sym(n))$.
The centralizer of $\Sym(k)$ is trivial for $k \geq 3$, so $\sigma = 1$.
This and \eqref{eq:lem:computing_centralizers:2} give that 
\begin{equation} \label{eq:lem:computing_centralizers:3}
    \text{$g_{\tau^{-1}} = h g h^{-1}$ for all $(h,\tau) \in H$.}
\end{equation}
We use \eqref{eq:lem:computing_centralizers:3} to prove that 
\begin{equation} \label{eq:lem:computing_centralizers:4}
    g_{\eta^{-1} \tau^{-1} \eta \tau} = g \
    \text{for all $\tau, \eta \in \Sym(n)$.}
\end{equation}
Since $\pi(H) = \Sym(n)$, there are $h,k \in G^n$ such that $(h,\tau)$ and $(k,\eta)$ belong to $H$.
Note that $(h,\tau)(k,\eta)$ is equal to $(h_{\eta^{-1}} k, \tau \eta)$, so, by \eqref{eq:lem:computing_centralizers:3},
\begin{equation*}
    g_{\eta^{-1} \tau^{-1}} = 
    g_{(\tau \eta)^{-1}} = 
    h_{\eta^{-1}} k g k^{-1} h_{\eta^{-1}}^{-1}.
\end{equation*}
Similarly, we can use \eqref{eq:lem:computing_centralizers:3} to compute as follows:
\begin{equation*}
    g_{\tau^{-1} \eta^{-1}} = 
    (g_{\tau^{-1}})_{\eta^{-1}} = 
    (h g h^{-1})_{\eta^{-1}} = 
    h_{\eta^{-1}} g_{\eta^{-1}} h_{\eta^{-1}}^{-1} = 
    h_{\eta^{-1}}  k g k^{-1} h_{\eta^{-1}}^{-1}.
\end{equation*}
We conclude that $g_{\eta^{-1} \tau^{-1}} = g_{\tau^{-1} \eta^{-1}}$ for all $\tau, \eta \in \Sym(n)$, and thus that \eqref{eq:lem:computing_centralizers:4} holds.

Equation \eqref{eq:lem:computing_centralizers:4} implies that $g_{[\tau, \eta]} = g$ for any $\tau, \eta \in \Sym(n)$.
Hence, by \eqref{eq:lem:computing_centralizers:1}, $g \in \Delta^n_G$.
We have proved that $(g,\sigma) \in \Delta^n_G \times \{1\}$, and thereby that $C((\Delta^n_{C(G)} \times \{1\}) \cdot H)$ is contained in $\Delta^n_G \times \{1\}$.

We now prove Item (2).
Notice that $C(\Delta^n_{C(G)} \times \Sym(n))$ is contained in $\Delta^n_G \times \{1\}$ by Item (2).
The reverse inclusion is a direct consequence of \eqref{eq:lem:computing_centralizers:2}.
\end{proof}

\begin{lemma} \label{lem:GisoHifTransAction}
Let $\phi \colon G\wr\Sym(n) \to H\wr\Sym(n)$ be an isomorphism, where $H$ is any group, and let $\pi \colon H\wr\Sym(n)\to\Sym(n)$ be the morphism onto the last coordinate.
Assume that $n \geq 3$ and that $\pi\circ\phi(\{1\}\times\Sym(n))$ acts transitively on $\{0,1,\dots,n-1\}$.
Then, $G$ is isomorphic to $H$.
\end{lemma}
\begin{proof}
On the one hand, Item (i) of \autoref{lem:computing_centralizers} gives that
\begin{equation} \label{eq:lem:GisoHifTransAction:1}
    C(\{1\} \times \Sym(n)) = \Delta^n_G \times \{1\} \cong G.
\end{equation}
On the other hand, as $\pi\circ\phi(\{1\}\times\Sym(n))$ acts transitively on $\{0,1,\dots,n-1\}$, $\phi(\{1\} \times \Sym(n))$ satisfies the hypothesis of Item (ii) of \autoref{lem:computing_centralizers}, so 
\begin{equation} \label{eq:lem:GisoHifTransAction:2}
    C(\phi(\{1\} \times \Sym(n))) = \Delta^n_H \times \{1\} \cong H.
\end{equation}
As $\phi$ bijectively maps $C(\{1\} \times \Sym(n))$ onto $C(\phi(\{1\} \times \Sym(n)))$, it follows from \eqref{eq:lem:GisoHifTransAction:1} and \eqref{eq:lem:GisoHifTransAction:2} that $G$ is isomorphic to $H$.
\end{proof}

We have all the elements to prove the two principal results of the section.

\begin{theorem} \label{prop:main_alg_lemma}
Let $G$ and $H$ be non-trivial groups, $n,m \geq 2$, and suppose that $G \wr \Sym(n)$ is isomorphic to $H \wr \Sym(m)$.
Then, $n = m$. 
Moreover, if $n \geq 4$, then $G$ is isomorphic to $H$.
\end{theorem}
\begin{remark}
When $n \in \{2,3\}$, our techniques only permit to prove that one of the groups $G$ and $H$ is a subgroup of index at most 2 of the other.
\end{remark}
% % % % % % % % % % PROOF PROOF % % % % % % % % % % 
% \begin{proof}[Proof of \autoref{prop:main_alg_lemma}] % revisar de nuevo
\begin{proof}
We start with some general observations.
We assume without loss of generality that $n \leq m$.
Let $G_j$ be the set of elements $g = (g_0,\dots,g_{n-1}) \in G^n$ such that $g_i = 1$ for all $i \in \{0,\dots,n-1\} \setminus \{j\}$.
We also define $K_j = \phi(G_j \times \{1\})$ and $K = \phi(G^n \times \{1\})$.
It follows from \autoref{lem:wreath_prod_formulas} that 
\begin{enumerate}[label=(\roman*')]
    \item $G_j$ is a normal subgroup of $G^n$,
    \item $G_0\cdots G_{n-1} \coloneqq \{g_0 g_1 \cdots g_{n-1} : g_j \in G_j\}$ is equal to $G^n$,
    \item $G_i \times \{1\}$ is conjugate to $G_j \times \{1\}$ in $G \wr \Sym(n)$ by the involution $\sigma(i) = j$, $\sigma(j) = i$, $\sigma(k) = k$ for $k \neq i,j$,
    \item $[G_i, G_j] \coloneqq \{ghg^{-1}h^{-1} : g \in G_i, h \in G_j\} = \{1\}$ if $i \not= j$.
\end{enumerate}
Hence, as $\phi$ is an isomorphism,
\begin{enumerate}[label=(\roman*)]
    \item $K_j$ is a normal subgroup of $K$,
    \item $K_0 \cdots K_{n-1} \coloneqq \{k_0 k_1 \cdots k_{n-1} : k_j \in K_j\}$ is equal to $K$,
    \item $K_i$ is conjugate to $K_j$ in $H \wr \Sym(m)$, and
    \item $[K_i, K_j] \coloneqq \{ghg^{-1}h^{-1} : g \in K_i, h \in K_j\} = \{1\}$ for all $i \not= j$.
\end{enumerate}
We denote by $\pi\colon H\wr\Sym(m) \to \Sym(m)$ the projection onto the last coordinate.
Recall that this is an onto morphism of groups.
Observe that $K \subseteq H^m \times \pi(K)$, so 
\begin{multline} \label{eq:prop:main_alg_lemma:1}
    [\Sym(m) : \pi(K)] = 
    [H\wr\Sym(m) : H^m \times \pi(K)] \\ \leq 
    [H\wr\Sym(m) : K] = 
    [G\wr\Sym(n) : G^n \times \{1\}] = |\Sym(n)|.
\end{multline}
Note that since $G^n\times\{1\}$ is normal in $G\wr\Sym(n)$ and $\pi \circ\phi$ is onto, $\pi(K)$ is normal in $\Sym(m)$.
\medskip

We now consider two cases.
Suppose that $\pi(K) = \{1\}$.
Then, by \eqref{eq:prop:main_alg_lemma:1}, $m \leq n$, so $m = n$.
Let us now further assume that $n \geq 3$ and prove that $G$ is isomorphic to $H$.
Since $G^n \times \{1\}$ and $\{1\} \times \Sym(n)$ generate $G \wr \Sym(n)$, their images under $\pi\circ\phi$ generate $\Sym(m)$.
Being $\pi\circ\phi(G^n \times \{1\})$ equal to $\pi(K) = \{1\}$, we deduce that $\pi \circ \phi(\{1\} \times \Sym(n)) = \Sym(m)$.
In particular, as $n = m$, the hypothesis of \autoref{lem:GisoHifTransAction} is satisfied, so $G$ is isomorphic to $H$.
\medskip

For the second case, we assume that $\pi(K) \not= \{1\}$.
First, we note that, by (i), (ii) and (iii),
\begin{equation} \label{eq:prop:main_alg_lemma:2}
    \{1\} \neq \pi(K_j) \trianglelefteq \pi(K) \trianglelefteq \Sym(m) \
    \text{for all $j \in \{0,1,\dots,n-1\}$.}
\end{equation}
Let $L = \phi(\{1\} \times \Sym(n))$.
We can compute 
\begin{multline*}
    \pi(L)\cdot\pi(K) = \pi(L\cdot K) =
    \pi\circ\phi(\{1\}\times\Sym(n) \cdot G^n\times\{1\}) \\ = 
    \pi\circ\phi(G\wr\Sym(n)) = \Sym(m).
\end{multline*}
Thus, by \autoref{lem:LisTransIfHasComplement},
\begin{equation} \label{eq:prop:main_alg_lemma:eitherAmOrTrans}
    \text{either $\pi(K) \supseteq A_m$ or $\pi(L)$ acts transitively on $\{0,1,\dots,m-1\}$.}
\end{equation}

\medskip

\textbf{Claim.}
The following holds:
\begin{equation} \label{eq:prop:main_alg_lemma:CLAIM}
    \text{$K$ is not equal to $H^m \times \pi(K)$.}
\end{equation}
\begin{proof}[Proof of the claim]
Assume, with the aim of obtaining a contradiction, that $K = H^m \times \pi(K)$.
Then, by (i), $(k,1) K_j (k,1)^{-1} \subseteq K_j$ for any $k \in H^m$ and $j \in \{0,1,\dots,n-1\}$.
This allows us to use \autoref{lem:fix_by_null} with $K_j$ and any $(h,\tau) \in K$, yielding the following:
If $\tau \in \pi(K)$ and we fix $j_p \in p$ for every $\tau$-orbit $p$, then 
\begin{equation}\label{eq:prop:main_alg_lemma:KjContainsInnerts}
    K_j \supseteq \{(g,1) \in H^m \times \{1\} : \text{$c_\tau(g,j_p) = 1$ for all $\tau$-orbit $p$}\}
\end{equation}
for every $j \in \{0,1,\dots,n-1\}$.

We now consider the following two cases.
\begin{enumerate}
    \item Assume that one of the groups $\pi(K_j)$ contains a normal subgroup $N$ of $\Sym(m)$ that is different from $\{1\}$.
    As (iii) ensures that $\pi(K_i)$ is conjugate to $\pi(K_j)$, $N$ is contained in each $\pi(K_j)$.
    Therefore, being $H$ a non-trivial group, \eqref{eq:prop:main_alg_lemma:KjContainsInnerts} implies that $K_i \cap K_j \neq \{1\}$ for every $i,j$.
    This contradiction shows that this case does not occur.

    \item Let us assume that every $\pi(K_j)$ does not contain a normal subgroup of $\Sym(m)$ that is different from $\{1\}$.
    Then, by \eqref{eq:prop:main_alg_lemma:2} and \autoref{lem:symmetric_groups_structure}, $m = 4$ and $\pi(K)$ is the normal subgroup $V$ of $\Sym(4)$ that is isomorphic to the Klein four-group.
    If we denote by $(i\ j)$ the element of $\Sym(4)$ that permutes $i$ and $j$, then we can describe $V$ as follows:
    \begin{equation*}
        V = \{1, (1\ 2)(3\ 4), (1\ 3)(2\ 4), (1\ 4)(2\ 3)\} \leq \Sym(4).
    \end{equation*}
    There is no loss of generality in assuming that $(1\ 2)(3\ 4) \in \pi(K_0)$.
    If $(1\ 2)(3\ 4)$ belongs to $\pi(K_1)$, then, by \eqref{eq:prop:main_alg_lemma:KjContainsInnerts}, $(h,1) \in K_0 \cap K_1$ for any $h \in H^4$ of the form $h = (\alpha, \alpha^{-1}, 1, 1)$, with $\alpha \in H$.
    This implies that $K_0 \cap K_1 \not= \{1\}$, which contradicts that the $G_j = \phi^{-1}(K_j)$ are disjoint.
    In turn, if $(1\ 3)(2\ 4)$ belongs to $K_1$, then \eqref{eq:prop:main_alg_lemma:KjContainsInnerts} gives that $(h,1) \in K_0 \cap K_1$ for any $h \in H^4$ of the form $h = (\alpha, \alpha^{-1}, \alpha^{-1}, \alpha)$, with $\alpha \in H$.
    This contradicts again that $K_0 \cap K_1 = \{1\}$.
    A similar argument works when $(1\ 4)(2\ 4) \in \pi(K_1)$.
    Thus, as $K_1 \subseteq V$, we have that $K_1 = \{1\}$, which is impossible by \eqref{eq:prop:main_alg_lemma:2}.
    This completes the proof of \eqref{eq:prop:main_alg_lemma:CLAIM}.
\end{enumerate}
\end{proof}

\textit{Proof of \autoref{prop:main_alg_lemma} continued.}
We now continue with the main part of the proof.
Assume first that $\pi(K)$ is not the Klein group $V$.
Then, \eqref{eq:prop:main_alg_lemma:2} and (iii) permit to use \autoref{lem:subgroups_with_=_card_sym_groups} to obtain that $\pi(K_i) = \pi(K_j)$ for all $i,j \in \{0,\dots,n-1\}$.
In view of (ii), we have that $\pi(K) = \pi(K_j)$ for every $j \in \{0,\dots,n-1\}$.
Combining this with (iv) yields that $[\pi(K), \pi(K)] = \{1\}$, that is, that $\pi(K)$ is abelian.
Given that $\pi(K)$ normal in $\Sym(m)$ by \eqref{eq:prop:main_alg_lemma:2}, \autoref{lem:symmetric_groups_structure} reveals that $\pi(K)$ can only be abelian if either $m = 2$ or $m = 3$ and $\pi(K) = A_3$.
If $m = 2$, then $n = m = 2$ (as we assumed that $n \leq m$) and we are done.
Let us assume that $m = 3$ and that $\pi(K) = A_3$.
Observe that, by \eqref{eq:prop:main_alg_lemma:1},
\begin{multline*}
    2 = [\Sym(m):A_m] = 
    [H\wr\Sym(m) : H^m\times\pi(K)] \\ \leq 
    [H\wr\Sym(m) : K] = |\Sym(n)| = n!
\end{multline*}
This implies that if $n$ were equal to $2$, then, since $H\wr\pi(K)$ has $K$ as a subgroup, $K$ would have to be equal to $H^m\times\pi(K)$, which contradicts \eqref{eq:prop:main_alg_lemma:CLAIM}.
Therefore, $n \neq 2$, and thus $n = m = 3$ (as we assumed that $n \leq m$).
The proof of this case is complete.
\medskip

We now consider the case when $\pi(K)$ is the Klein group $V$.
Note that, by \eqref{eq:prop:main_alg_lemma:2} and \autoref{lem:symmetric_groups_structure}, $m$ must be equal to $4$.
Hence, as the index of $V$ in $\Sym(m)$ is $6$, we have from \eqref{eq:prop:main_alg_lemma:1} that 
\begin{equation} \label{eq:prop:main_alg_lemma:3.1}
    6 = [\Sym(m):\pi(K)] = 
    [H\wr\Sym(m) : H^m\times\pi(K)] \leq 
    [H\wr\Sym(m) : K] = n!
\end{equation}
Being $n \leq m = 4$, either $n = 3$ or $n = 4$.
If $n = 3$, then \eqref{eq:prop:main_alg_lemma:3.1} implies that $K = H^m \times \pi(K)$, which is impossible by \eqref{eq:prop:main_alg_lemma:CLAIM}.
Hence, $n = m = 4$.
Notice that \eqref{eq:prop:main_alg_lemma:eitherAmOrTrans} ensures, as $\pi(K) = V \subsetneq A_4$, that $\pi(L)$ acts transitively on $\{0,1,2,3\}$, where $L = \phi(\{1\}\times\Sym(4))$.
So, by \autoref{lem:GisoHifTransAction}, $G$ is isomorphic to $H$.
\end{proof}

% % % % % % % % % % SUBSECTION % % % % % % % % % % 
% \subsection{Proof of \texorpdfstring{\autoref{prop:main_alg_lemma2}}{}}
% \label{subsec:proof_2°_alg_lemma}

The last result of this section is an analog of \autoref{prop:main_alg_lemma} applicable when one of the groups under consideration is approximated by, rather than isomorphic to, wreath products of the form $H_k \wr \Sym(n_k)$.
This scenario arises naturally in the study of stabilized automorphism groups of transitive systems.

\begin{proposition} \label{prop:main_alg_lemma2}
Let $G, H$ be groups and $n \geq 2$.
Suppose that there are subgroups $(H_\ell \leq H : \ell \geq 1)$ of $H$ such that:
\begin{enumerate}
    \item $H_\ell \leq H_{\ell+1}$ for $\ell \geq 1$ and $H = \bigcup_{\ell\geq1} H_\ell$.
    \item For all $\ell \geq 1$, there is an isomorphism $H_\ell \cong H'_\ell \wr \Sym(m_\ell)$ for certain group $H'_\ell$ and $m_\ell \geq 1$.
    \item For all $r \geq 1$ there are infinitely many $\ell \geq r$ such that $\phi_\ell^{-1}(H'_\ell \times \{1\})$ is contained in $\phi_r^{-1}(H'_r \times \{1\})$.
\end{enumerate}
If $G \wr \Sym(n)$ is isomorphic to $H$, then $\limsup_{\ell\to+\infty} m_\ell < +\infty$.
\end{proposition}
\begin{proof}
Let $\phi\colon G \wr \Sym(n) \to H$ and $\phi_{\ell} \colon H_{\ell} \to H'_{\ell} \wr \Sym(m_{\ell})$ be isomorphisms.
We denote by $G_j$ the subgroup of $G^n$ consisting of vectors $g = (g_0,\dots,g_{n-1})$ such that $g_i = 1$ if $i \in \{0,1,\dots,n-1\} \setminus \{j\}$.
We set $K = \phi(G^n \times \{1\})$ and $K_j = \phi(G_j \times \{1\})$.
Note that $G_j \times \{1\} \trianglelefteq G^n \times \{1\} \trianglelefteq G \wr \Sym(n)$.
Hence, if we define 
\begin{equation*}
    \text{$K'_{\ell} = \phi_{\ell}(K \cap H_{\ell})$ and $K'_{j,\ell} = \phi_{\ell}(K_j \cap H_{\ell})$,}
\end{equation*}
then
\begin{equation} \label{eq:prop:main_alg_lemma2:nested_normal_subgroups}
    K'_{j,\ell} \trianglelefteq K'_{\ell} \trianglelefteq H'_{\ell} \wr \Sym(m_{\ell})
    \enskip \text{for all $\ell \geq 1$.}
\end{equation}
Let $\pi_{\ell}\colon H'_{\ell} \wr \Sym(m_{\ell}) \to \Sym(m_{\ell})$ the projection onto the last coordinate.

We suppose, with the aim of obtaining a contradiction, that 
\begin{equation} \label{eq:prop:main_alg_lemma2:contra_hip}
    \limsup_{\ell\to+\infty} m_{\ell} = +\infty.
\end{equation}
The proof continues with two cases.

Assume first that $\pi_{\ell}(K'_{\ell}) = \{1\}$ for all $\ell$ belonging to an infinite set $F \subseteq \N$.
Then, for such values of $\ell$,
\begin{equation} \label{eq:prop:main_alg_lemma2:0}
    [H'_{\ell} \wr \Sym(m_{\ell}) : K'_{\ell} ] \geq 
    [H'_{\ell} \wr \Sym(m_{\ell}) : {H'}_{\ell}^{m_{\ell}} \times \{1\}] = 
    m_{\ell}!
\end{equation}
Since $[G \wr \Sym(n): G^n \times\{1\}] = n!$, there are elements $g_1,\dots,g_{n!} \in G \wr \Sym(n)$ such that the cosets $g_j (G^n\times\{1\})$ form a partition of $G \wr \Sym(n)$.
Then, as $\phi$ is an isomorphism,
\begin{equation} \label{eq:prop:main_alg_lemma2:1}
    \text{the cosets $\{\phi(g_j) K : j \in \{1,\dots,n!\}\}$ form a partition of $H$.}
\end{equation}
Using \eqref{eq:prop:main_alg_lemma2:contra_hip} and Item (1) of the hypothesis, we can find $\ell \in F$ big enough so that $m_{\ell} > n$ and $\phi(g_j) \in H_{\ell}$ for every $j \in \{1,\dots,n!\}$.
The last condition and \eqref{eq:prop:main_alg_lemma2:1} imply that 
\begin{equation*}
    \{\phi(g_j) (K \cap H_{\ell}) : 1 \leq j \leq n!\}
    \ \text{form a partition of $H_{\ell}$.}
\end{equation*}
Given that $\phi_{\ell}$ is an isomorphism, we get that $[H'_{\ell} \wr \Sym(m_{\ell}) : K'_{\ell} ] = n! < m_{\ell}!$, which contradicts \eqref{eq:prop:main_alg_lemma2:0}.

\medskip
Let us now assume that
\begin{equation} \label{eq:prop:main_alg_lemma2:hip_last_case}
    \text{$\pi_{\ell}(K'_{\ell}) \neq \{1\}$ for every large enough $\ell$.}
\end{equation}
Equation \eqref{eq:prop:main_alg_lemma2:nested_normal_subgroups} implies that 
\begin{equation} \label{eq:prop:main_alg_lemma2:normal_chain}
    \pi_{\ell}(K'_{j,\ell}) \trianglelefteq \pi_{\ell}(K'_{\ell}) \trianglelefteq \Sym(m_{\ell})
    \enskip \text{for $j \in \{0,\dots,n-1\}$ and large $\ell$.}
\end{equation}
Hence, if $\ell$ is large enough so that \eqref{eq:prop:main_alg_lemma2:normal_chain} holds and $m_{\ell} \geq 5$, then \autoref{lem:symmetric_groups_structure} guarantees that, for every $j \in \{0,\dots,n-1\}$,
\begin{equation} \label{eq:prop:main_alg_lemma2:3}
    \text{$\pi_\ell(K'_{j,\ell})$ is equal to $\{1\}$, $A_{m_\ell}$ or $\Sym(m_\ell)$.}
\end{equation}
where $A_{m_\ell}$ is the alternating group on $m_\ell$ elements.
Now, since $G_i \times \{1\}$ is conjugate to $G_j \times \{1\}$ in $G \wr \Sym(n)$, for every $i,j \in \{0,\dots,n-1\}$ there is $\eta_{i,j} \in H$ such that $K_j = \eta_{i,j} K_i \eta_{i,j}^{-1}$.
Let $\ell \geq 1$ be any integer big enough so that $\eta_{i,j} \in H_{\ell}$ for all $i,j \in \{0,\dots,n-1\}$.
This condition ensures that $K'_{j,\ell} = \phi_{\ell}(\eta_{i,j}) K'_{i,j} \phi_{\ell}(\eta_{i,j})^{-1}$, and thus that 
\begin{equation} \label{eq:prop:main_alg_lemma2:4}
    \text{$\pi_{\ell}(K'_{i,\ell})$ is conjugate to $\pi_{\ell}(K'_{j,\ell})$ in $\Sym(m_{\ell})$ for all $i,j \in \{0,\dots,n-1\}$.}
\end{equation}
In view of \eqref{eq:prop:main_alg_lemma2:3}, we deduce that, for any large enough $\ell$,
\begin{equation} \label{eq:prop:main_alg_lemma2:5}
    \text{$\pi_{\ell}(K'_{i,\ell}) = \pi_{\ell}(K'_{j,\ell})$ for all $i,j \in \{0,\dots,n-1\}$.}
\end{equation}
Note that, since the elements of $G_i \times \{1\}$ commute with those of $G_j \times \{1\}$ if $i \neq j$, we have that the elements of $\pi_\ell(K'_{i,\ell})$ commute with those of $\pi_\ell(K'_{j,\ell})$.
Therefore, by \eqref{eq:prop:main_alg_lemma2:5}, the groups $\pi_{\ell}(K'_{j,\ell})$ are abelian for large $\ell$.
Combining this with \eqref{eq:prop:main_alg_lemma2:3} and the fact that $A_m$ is non-abelian when $m \geq 4$ yields that
\begin{equation} \label{eq:prop:main_alg_lemma2:6}
    \pi_{\ell}(K'_{j,\ell}) = \{1\}
    \enskip \text{for $j \in \{0,\dots,n-1\}$ and large $\ell$.}
\end{equation}
We use \eqref{eq:prop:main_alg_lemma2:6} to obtain a contradiction.

It follows from \eqref{eq:prop:main_alg_lemma2:hip_last_case} and \eqref{eq:prop:main_alg_lemma2:6} that there is a integer $r \geq 1$ such that $\phi_r(K'_r) \neq \{1\}$ and $\pi_{\ell}(K'_{j,\ell}) = \{1\}$ for all $\ell \geq r$ and $j \in \{0,\dots,n-1\}$.
The first condition permits to take $h \in K \cap H_r$ satisfying $\pi_r \phi_r(h) \neq 1$.
Now, since $G = G_0 G_1 \dots G_{n-1}$, we have that $K = K_0 K_1 \dots K_{n-1}$, and thus, by Item (1) of the hypothesis, that there exist $\ell \geq r$ and elements $h_j \in K_j \cap H_\ell$ such that $h = h_0 h_1 \dots h_{n-1}$.
We take, using Item (3) of the hypothesis, $s \geq \ell$ such that $\phi_s^{-1}(H'_s \times \{1\})$ is contained in $\phi_r^{-1}(H'_r \times \{1\})$.
Given that $s \geq r$, $\pi_{s}(K'_{j,s}) = \{1\}$ for every $j \in \{0,\dots,n-1\}$, so
\begin{equation*}
    \pi_s\phi_s(h) = 
    \pi_s\phi_s(h_{0}) \pi_s\phi_s(h_{1}) \dots \pi_s\phi_s(h_{n-1}) = 1.
\end{equation*}
We deduce, by the definition of $s$, that $h \in \phi_s^{-1}(H'_s \times \{1\}) \subseteq \phi_r^{-1}(H'_r \times \{1\})$.
This implies that $\pi_r\phi_r(h) = 1$, contradicting the choice of $h$.
\end{proof}

\section{Recovering rational eigenvalues}
\label{sec:recover_eigs}
We present in this section our results on recovering rational eigenvalues from the stabilized automorphism group.
We start by proving \autoref{theo:recover_eigs:abstract} and \autoref{theo:recover_eigs:low_comp:abstract}.
Then, in \autoref{subsec:rec_eigs:trans-case}, two examples are constructed to show that the hypothesis of \autoref{theo:recover_eigs:abstract} cannot be weakened.
% \begin{theorem} \label{theo:recover_eigs:abstract}
% Let $(X,T)$ and $(Y,S)$ be minimal systems with at least one rational eigenvalue different from $1$.
% If $\StabAut(T)$ is isomorphic to $\StabAut(S)$, then the systems have the same rational eigenvalues.
% \end{theorem}
\begin{proof}[Proof of \autoref{theo:recover_eigs:abstract}]
Let $\phi \colon \StabAut(T) \to \StabAut(S)$ be an isomorphism and let $n \in \Eig(T) \setminus \{1\}$.
By symmetry, it is enough to show that $n \in \Eig(S)$.
We take $\ell \geq 1$ big enough so that $\phi^{-1}(S) \in \Aut(T^\ell)$, and set $\gamma = \phi(T)$.
Observe that, since $\phi^{-1}(S)$ commutes with $T^\ell$,
\begin{equation} \label{eq:theo:recover_eigs:abstract:gamma_ell_in_Aut_S}
    \text{$\gamma^\ell$ belongs to $\Aut(S)$.}
\end{equation}
We fix $m \in \Eig(S)$, with $m \geq 2$, and define $N = nm\ell$.
Let us write, using \autoref{decompose:trans_peig}, $k!N = p_k m_k$, where $S^{p_k}$ acts transitively on $Y$ and $m_k \in \CAP(S)$.
Being $(Y,S)$ minimal, \autoref{minimal_sofic=>cP} ensures that $m_k \in \Eig(S)$.
Then, in view of \eqref{eq:theo:recover_eigs:abstract:gamma_ell_in_Aut_S}, we can use \autoref{prop:C(gamma)_as_wreathProd} with $\gamma^N$, $p_k$ and $m_k$ to obtain $j_k \in \{0,\dots,m_k-1\}$ such that $\gamma(X_{m_k}) = S^{j_k} X_{m_k}$ and, if $r_k = \gcd(j_k, m_k)$, then there is an isomorphism
\begin{equation} \label{eq:theo:recover_eigs:abstract:defi_phi_k}
    \phi_k \colon C_{\Aut(S^{k!N})}(\gamma^N) \to 
    H_k \wr \Sym(r_k),
\end{equation}
where $H_k = C_{\Aut(S^{k!N}|_{Y_{r_k}})}(\gamma^N)$ and $Y_{r_k} \subseteq Y$ defines a cyclic partition of size $r_k$.
\medskip

We continue with some observations about $\Aut(T^N)$.
Thanks to \autoref{decompose:trans_peig}, we can write $N = pq$, where $T^p$ acts transitively on $X$ and $q \in \CAP(T)$.
Note that, by \autoref{transpowers}, $\gcd(n, p) = 1$.
Hence, since $N = nm\ell = pq$, we deduce that
\begin{equation} \label{eq:theo:recover_eigs:abstract:n_divides_q}
    \text{$n$ divides $q$; in particular, $q \geq n \geq 2$.}
\end{equation}
Also, as $(X,T)$ is minimal, \autoref{minimal_sofic=>cP} guarantees that $\Eig(T) = \CAP(T)$, so $q \in \Eig(T)$.
This permits to use \autoref{semidirectproduct} and write 
\begin{equation*}
    \text{$\Aut(T^N) \cong \Aut(T^N|_{X_q}) \wr \Sym(q)$,}
\end{equation*}
where $X_q \subseteq X$ defines a cyclic partition of size $q$.
Note that, since $\Aut(T^N)$ is the centralizer of $T^N$, $\phi$ maps $\Aut(T^N)$ onto $C(\gamma^N)$.
Thus,
\begin{equation} \label{eq:theo:recover_eigs:abstract:iso_between_C}
    \Aut(T^N|_{X_q}) \wr \Sym(q) \cong
    \Aut(T^N) \cong C(\gamma^N).
\end{equation}
\medskip

We now consider two cases.
Assume first that $(r_k)_{k\geq1}$ is bounded.
Then, there exist $r \geq 1$ and an infinite set $F \subseteq \N$ such that $r_k = r$ for all $k \in F$.
Observe that if $\pi_k\colon H_k \wr \Sym(r_k) \to \Sym(r_k)$ is the factor onto the last coordinate, then \autoref{prop:C(gamma)_as_wreathProd} ensures that, for any $k \in F$,
\begin{multline} \label{eq:theo:recover_eigs:abstract:g_permutes}
    \text{$g(S^i Y_r) = S^{\sigma(i)} Y_r$ whenever
        $g \in C_{\Aut(S^{k!N})}(\gamma^N)$, } \\ \text{ 
        $i \in \{0,\dots,m_k-1\}$ and $\sigma = \pi_k\phi_k(g)$.}
\end{multline}
This implies that, if $k \leq k'$ are in $F$, then the restriction of $\pi_{k'}\phi_{k'}$ to $C_{\Aut(S^{k!N})}(\gamma^N)$ is equal to $\pi_k\phi_k$.
For these values of $k$ and $k'$, we also have that
\begin{equation*}
    H_k = C_{\Aut(S^{k!N}|_{Y_r})}(\gamma^N) \subseteq 
    C_{\Aut(S^{(k+1)!N}|_{Y_r})}(\gamma^N) = H_{k+1}.
\end{equation*}
These two things imply, as the $\phi_k$ are isomorphisms, that
\begin{equation} \label{eq:theo:recover_eigs:abstract:C_gamma_N_as_wreath}
    C(\gamma^N) = \bigcup_{k\geq1} C_{\Aut(S^{k!N})}(\gamma^N)
    \cong H \wr \Sym(r), \
    \text{where $H = \bigcup_{k\in F} H_k$.}
\end{equation}
Putting this in \eqref{eq:theo:recover_eigs:abstract:iso_between_C} yields
\begin{equation} \label{eq:theo:recover_eigs:abstract:iso_between_wreaths}
    \text{$\Aut(T^N|_{X_q}) \wr \Sym(q)$ is isomorphic to $H \wr \Sym(r)$.}
\end{equation}
With the aim of using \autoref{prop:main_alg_lemma}, let us prove that $q, r \geq 2$.
We know, from \eqref{eq:theo:recover_eigs:abstract:n_divides_q}, that $q \geq 2$.
For $r$, we first recall that \eqref{eq:theo:recover_eigs:abstract:gamma_ell_in_Aut_S} states that $\gamma^\ell \in \Aut(S)$.
Hence, we can use \autoref{prop:C(gamma)_as_wreathProd} with $\gamma^\ell$ and $m_k \in \Eig(S)$ to obtain $j'_k \in \{0,\dots,m_k-1\}$ such that $\gamma^\ell(Y_{m_k}) = S^{j'_k} Y_{m_k}$ for every $k \in F$.
Hence, by the definition of $j_k$,
\begin{equation*}
    S^{j_k} Y_{m_k} = \gamma^N(Y_{m_k}) = 
     (\gamma^\ell)^{N/\ell} (Y_{m_k}) = 
    S^{j'_k \cdot N/\ell} Y_{m_k}.
\end{equation*}
This implies that $j_k = j'_k (N/\ell) \pmod{m_k}$.
Now, since we decomposed $N\cdot k! = nm\ell\cdot k!$ as $p_k m_k$ using \autoref{decompose:trans_peig} and since $m \in \Eig(S)$, we have that $m$ divides $m_k$.
Thus, $j_k = j'_k (N/\ell) \pmod{m}$.
Moreover, since $N = nm\ell$, $m$ divides $N/\ell$ and thus $j_k = 0 \pmod{m}$.
We conclude that $r_k = \gcd(j_k, m_k) \geq m \geq 2$, and thus that $r \geq 2$.

We have proved that $q, r \geq 2$.
This and \eqref{eq:theo:recover_eigs:abstract:iso_between_wreaths} allow us to apply \autoref{prop:main_alg_lemma}; so $q = r$.
We get that $q = r \in \Eig(S)$, and then, as $n$ divides $q$ by \eqref{eq:theo:recover_eigs:abstract:n_divides_q}, that $n \in \Eig(S)$.

\medskip
We now consider the second case.
Let us assume that $(r_k)_{k\geq1}$ is unbounded.
Our strategy is to use \autoref{prop:main_alg_lemma2} with $\Aut(T^N)$ and $C(\gamma^N)$.
The groups $\Aut(T^N)$ and $C(\gamma^N)$ are isomorphic by \eqref{eq:theo:recover_eigs:abstract:iso_between_C}.
We proved, in \eqref{eq:theo:recover_eigs:abstract:iso_between_C}, that $\Aut(T^N)$ has the form $G \wr \Sym(q)$, with $q \geq 2$.
From \autoref{eq:theo:recover_eigs:abstract:defi_phi_k} we have a description of $C(\gamma^N)$ as a non-decreasing union of the groups $C_{\Aut(S^{k!N})}(\gamma^N)$, which are isomorphic to the wreath products $H_k \wr \Sym(r_k)$.
It only rests to check Item (3) of the hypothesis of \autoref{prop:main_alg_lemma2}.
It is enough to prove that for $t \geq s \geq 1$, we have that $\phi_t^{-1}(H_t \times \{1\})$ is contained in $\phi_s^{-1}(H_s \times \{1\})$.
Let $g \in \phi_t^{-1}(H_t \times \{1\})$.
Then, by the description of $\phi_t$ from \autoref{prop:C(gamma)_as_wreathProd},
\begin{equation}  \label{eq:theo:recover_eigs:abstract:is_in_kernel}
    \text{$g(S^i Y_{r_t}) = S^i Y_{r_t}$ for every $i \in \{0,1,\dots,r_t-1\}$.}
\end{equation}
We have, by the definition of the $m_k$ and by \autoref{decompose:trans_peig}, that $m_s$ divides $m_t$.
Now, since $m_s, m_t \in \Eig(S)$, \autoref{CycPart2divisor} permits to assume that
\begin{equation*}
    Y_{m_s} = \bigcup_{0 \leq i < m_t/m_s} S^{im_s} Y_{m_t}.
\end{equation*}
By applying $\gamma$ and using the definition of $j_s$ and $j_t$, we get that 
\begin{equation*}
    S^{j_s} Y_{m_s} = 
    \bigcup_{0 \leq i < m_t/m_s} S^{im_s+j_t} (Y_{m_t})
    = S^{j_t} Y_{m_s}.
\end{equation*}
Hence, as $Y_{m_s}$ is $S^{m_s}$-invariant, $j_s = j_t \pmod{m_s}$.
This and the fact that $m_s$ divides $m_t$ imply, by the definition of the $r_k$, that $r_s$ divides $r_t$.
Then, as we argued for $Y_{m_s}$ and $Y_{m_t}$, we may assume that 
\begin{equation*}
    Y_{r_s} = \bigcup_{0 \leq i < r_t/r_s} S^{ir_s} (Y_{r_t}).
\end{equation*}
From this and \eqref{eq:theo:recover_eigs:abstract:is_in_kernel} we deduce that $g(S^i Y_{r_s}) = Y_{r_s}$ for every $i \in \Z$, that is, that $g$ belongs to $\phi_s^{-1}(H_s \times \{1\})$.

% It is enough to prove that for $t \geq s \geq 1$, we have that $\phi_t^{-1}(H_t \times \{1\})$ is contained in $\phi_s^{-1}(H_s \times \{1\})$.
% Let $g \in \phi_t^{-1}(H_t \times \{1\})$.
% Then, by the description of $\phi_t$ from \autoref{prop:C(gamma)_as_wreathProd},
% \begin{equation}  \label{eq:theo:recover_eigs:abstract:is_in_kernel}
%     \text{$g(S^i Y_{r_t}) = S^i Y_{r_t}$ for every $i \in \{0,1,\dots,r_t-1\}$.}
% \end{equation}
% We have, by the definition of the $m_k$ and by \autoref{decompose:trans_peig}, that $m_s$ divides $m_t$.
% Now, note that, since $m_t \in \Eig(S)$, $Y_{m_t} \cup S(Y_{m_t}) \cup \dots \cup S^{m_t/m_s-1}(Y_{m_t})$ defines a cyclic partition of size $m_s$.
% Therefore, we may assume that
% \begin{equation*}
%     Y_{m_s} = \bigcup_{0 \leq i < m_t/m_s} S^i (Y_{m_t}).
% \end{equation*}
% By applying $\gamma$ and using the definition of $j_s$ and $j_t$, we get that 
% \begin{equation*}
%     S^{j_s} Y_{m_s} = 
%     \bigcup_{0 \leq i < m_t/m_s} S^{i+j_t} (Y_{m_t})
%     = S^{j_t} Y_{m_s}.
% \end{equation*}
% Hence, as $Y_{m_s}$ is $S^{m_s}$-invariant, $j_s = j_t \pmod{m_s}$.
% This and the fact that $m_s$ divides $m_t$ imply, by the definition of the $r_k$, that $r_s$ divides $r_t$.
% Then, as we argued for $Y_{m_s}$ and $Y_{m_t}$, we may assume that 
% \begin{equation*}
%     Y_{r_s} = \bigcup_{0 \leq i < r_t/r_s} S^i (Y_{r_t}).
% \end{equation*}
% From this and \eqref{eq:theo:recover_eigs:abstract:is_in_kernel} we deduce that $g(S^i Y_{r_s}) = Y_{r_s}$ for every $i \in \Z$, that is, that $g$ belongs to $\phi_s^{-1}(H_s \times \{1\})$.

The hypothesis of \autoref{prop:main_alg_lemma2} are satisfied, so $\limsup_{k\to+\infty} r_k$ is finite.
This contradicts that $(r_k)_{k\geq1}$ is unbounded, and thus completes the proof of the theorem.
\end{proof}

% We can restate \autoref{theo:recover_eigs:low_comp:abstract} as follows.
% \begin{theorem} \label{theo:recover_eigs:low_comp}
% Let $(X,T)$ and $(Y,S)$ be minimal systems such that $\StabAut(T)$ is isomorphic to $\StabAut(S)$.
% Assume that both $\Aut_{M(T)}(T)$ and $\Aut_{M(S)}(S)$ are abelian or virtually $\Z$.
% Then, the systems have the same rational eigenvalues.
% \end{theorem}
\begin{proof}[Proof of \autoref{theo:recover_eigs:low_comp:abstract}]
Thanks to \autoref{theo:recover_eigs:abstract} and a symmetry argument, we only have to consider the case in which $(X,T)$ does not have rational eigenvalues different from $1$.
We argue by contradiction and assume that $\Eig(S)$ contains an element $m$, with $m \geq 2$.
Let $\phi\colon\StabAut(T) \to \StabAut(S)$ be an isomorphism.
Then, by \autoref{semidirectproduct}, $\StabAut(S)$ has a subgroup isomorphic to $\Aut(S^m|_{Y_m}) \wr \Sym(m)$.
In particular, $\StabAut(S)$ contains a copy of $\Z \wr \Sym(m)$.
Now, being $(X,T)$ minimal, $\Eig(T) = \CAP(T)$ by \autoref{minimal_sofic=>cP}.
Hence, as $\Eig(T) = \{1\}$, and \autoref{transpowers} imply that $\StabAut(T) = \Aut_{M(T)}(T)$.
We deduce that 
\begin{equation} \label{eq:theo:recover_eigs:low_comp:has_big_subgroup}
    \text{$\Aut_{M(T)}(T)$ contains a copy of $\Z \wr \Sym(m)$.}
\end{equation}
Since $\Z \wr \Sym(m)$ is non-abelian (as $m \geq 2$), \autoref{eq:theo:recover_eigs:low_comp:has_big_subgroup} gives rise to a contradiction when $\Aut_{M(T)}(T)$ is abelian.
Moreover, $\Z \wr \Sym(m)$ has a copy of $\Z^2$, so \autoref{eq:theo:recover_eigs:low_comp:has_big_subgroup} produces a contradiction if $\Aut_{M(T)}(T)$ is virtually $\Z$. 
\end{proof}

\begin{corollary} \label{theo:eigs_Toeplitz}
Let $(X,T)$ and $(Y,S)$ be Toeplitz subshifts.
If their stabilized automorphism groups are isomorphic, then the systems have the same odometer as the maximal equicontinuous factor.
\end{corollary}
\begin{proof}
It is known that $(X,T)$ is minimal, has infinitely many rational eigenvalues, and $\Aut(T)$ is abelian.
Moreover, it was shown in \cite{jones} that $\Aut_{M(T)}(T) = \Aut(T)$.
The same considerations hold for $(Y,S)$.
So, the result follows from \autoref{theo:recover_eigs:abstract}.
\end{proof}

\begin{corollary} \label{theo:eigs_odometers}
Let $(X,T)$ and $(Y,S)$ be odometers.
If their stabilized automorphism groups are isomorphic, then the systems are conjugate.
\end{corollary}
\begin{proof}
Any odometer is minimal and has infinitely many rational eigenvalues.
Hence, we can use \autoref{theo:recover_eigs:abstract} to deduce that $(X,T)$ and $(Y,S)$ have the same rational eigenvalues.
An odometer does not have irrational eigenvalues, so these systems have the same spectrum.
As they are equicontinuous, it follows that $(X,T)$ is conjugate to $(Y,S)$.
\end{proof}

\subsection{The transitive case}
\label{subsec:rec_eigs:trans-case}
% % % % % % % % % % SECTION % % % % % % % % % %
In this subsection, we elucidate, through two examples, the shortcomings of the methods employed to establish \autoref{theo:recover_eigs:abstract} when the minimality hypothesis is relaxed to transitivity. The first example \ref{example1} exhibits a system $(X,T)$ with no rational eigenvalues but such that $(X,T)$ has a cyclic almost-partition of size $2^n$ for all $n>0$. In this example, the a cyclic almost-partition of size $q$ have the same effect in the stabilized automorphism group as an eigenvalues equal to $1/q$ would in the minimal case. Namely, on powers equal to $2^n$ for $n>0$, $\Aut(X,T^{2^n})$ is a wreath product with $\Sym(2^n)$. One might anticipate that in transitive systems, given that the condition $\Eig(T) = \CAP(T)$ is not necessarily satisfied, comparable outcomes to those in the minimal case could be attained, albeit with $(X,T)$ has a cyclic almost-partition instead of eigenvalues. However, this is not the case. Our second example \ref{example2}, presents a system $(X,T)$ that has a cyclic almost-partition of size $3^n$ for all $n>0$ but with trivial stabilized automorphism group.

\subsubsection{Example 1} \label{example1}
    This is an example of a transitive system $(X,T)$ with no eigenvalues, such that $T^{2^n}$ is not transitive and $\Aut(X,T^{2^n})$ is isomorphic to a wreath product with $\Sym(2^n)$.

    Let $b_n=10^{2^n-1}$ and $A_0$ be the empty word. Define $A_n$ iteratively as $A_{n}=A_{n-1}A_{n-1}b_n$ and let $\hat{A}$ to be the infinite word defined by the limit of this recursive relation. Take $u=(u_i) \in \{0,1\}$ defined as $u_i=0$ for all $i<0$ and $u_{[0,\infty)}=\hat{A}$.  Consider $T$ the shift and $X=\overline{\mathcal{O}_T(u)}$. Notice $(X,T)$ is transitive, and the point $\overline{0}=\{0\}^\Z\in X$. Since $\overline{0}$ is a fixed point of $(X,T)$, we have that $(X,T)$ has no eigenvalues.

    For all $n\geq 1$, notice that $b_{n+1}=b_n0^{2^{n}}$. Hence, by the recursion formula, all $x\in \mathcal{O}(u)$, $x$ can be written uniquely
    %{\color{red} uniquely ?} 
    as a sequence with elements in $\{A_{n-1}, b_n, 0^{2^{n}}\}$. Since $|b_n|=2^n$, by the recursion formula $|A_{n-1}|=2|A_{n-2}|+|b_n|$, we have that $|A_n|$ is a multiple of $2^n$. This implies that every element in $\mathcal{O}(u)$ belongs to one and only one of the following sets:
    $$X_k=\{x\in X: \text{if } x_{[i, i+2^{n}-1]}=b_n \text{ then } i= k\pmod{ 2^n}\}, \quad \text{ for } k=0,1,...,2^n-1.$$
    It is easy to verify that $X_k$ is closed for $k=0,1,...,2^n-1$.  We now show $X=\bigcup_{k=0}^{2^n-1} X_k$ and $X_i \cap X_j=\{\overline{0}\}$ for $i\neq j$. It is clear that $u\in X_0$ and $T^m(u)\in X_{m \pmod {2^n}}$ for all $n\in \N$. Take $x=(x_i)\in X$, then there exists a sequence $(m_j)$ such that $T^{m_j}(u)$ converges to $x$. Hence, for all $M\in \N$, there exits  $j\in \N$ such that $T^{n_j}(u)_i=x_i$ for all $i$ with $|i|\leq M$. Since each element in the sequence $(T^{n_j}(u))$ belongs to a set $X_k$, this shows that the indices of all occurrences of $b_n$ in $x_{[-M,M]}$ occur at the same number modulo $2^n$. Since this is true for all $M\in \N$, we conclude that $x\in \bigcup_{k=0}^{2^n-1} X_k$. Moreover, this shows that for all $x=(x_i)\in X\setminus\{\overline{0}\}$, if $(m_j)$ is a sequence such that $T^{m_j}(u)$ converges to $x$, then there exists $K\in \N$ such that $T^{n_j}(u)\in X_k$ for all $j\geq K$ and some $k\in \{0,1,...,2^n-1\}$ fixed. That is, there exists $K\in \N$ such that $m_j=m_K\pmod 2^n$ for all $j\geq K$. We refer to this as property $(\star)$. Notice that this property also implies that $X_i \cap X_j=\{\overline{0}\}$.

    To show that $\Aut(X,T^{2^n})$ contains a wreath product with $\Sym(2^n)$ we first prove the existence of a surjective homomorphism $\iota\colon\Aut(X,T^{2^n})\to \Sym(2^n)$. It is clear that $T^{2^n}$ has $\overline{0}$ as a fixed point and leaves $X_k$ invariant for $k=1,2,...,2^n-1$. Let $g\in \Aut(X,T^{2^n})$. Take $x\in X_0\setminus\{\overline{0}\}$, by property $(\star)$ then there exists a sequence $(n_i)$ such that $n_i=0\pmod 2$ and $T^{n_i}(u)$ converges to $x$. Since $T$ and $g$ commute and $T^{2^n}$ leaves $X_k$ invariant for $k=1,2,...,2^n-1$, this implies that if $g(u)\in X_j$ then $g(x)\in X_j$ for $j=0,1,...,2^n$. This shows that if $g(u)\in X_j$, then $g(X_0)\subseteq X_j$. Equivalently, we can show that if $g(T^k(u))\in X_j$, then $g(X_k)\subseteq X_j$ for all $k=1,...,2^n$. Since $g$ is a bijection on $X$, this shows that $g$ defines a permutation of the sets ${X_0, X_1,...,X_{2^n-1}}$. This implies the existence of a map $\iota\colon\Aut(X,T^{2^n})\to \Sym(2^n)$.

    We now show $\iota$ is surjective. Notice that for all $i,j\in \{0,1,...,2^n-1\}$ we have that $t_{j,i}:=T^{i-j}$ defines a conjugacy between $(X_j, T^{2^n}|_{X_j})$ and $(X_i, T^{2^n}_{X_i})$. Let $\pi\in \Sym(2^n)$. Define $\phi_\pi\in \Aut(X,T^{2^n})$ such that $X_i$ is mapped to $X_{\pi(i)}$ via $t_{i,\pi(i)}$. Since the sets ${X_0, X_1,...,X_{2^n-1}}$ are closed and the only point in their intersection is fixed by all $t_{i,j}$, $\phi_\pi$ is continuous. It is clear that $\phi_\pi$ commutes with $T^{2^n}$. We can conclude that $\Phi_\pi$ is an automorphism of $(X,T^{2^n})$. Notice $\iota(\phi_\pi)=\pi$. Thus $\iota$ is surjective. It is easy to verify that $\iota$ is a group homomorphism.

    For all $k\in\{0,1,...,2^n-1\}$ fix $g_k\in \Aut (X_k,T^{2^n}|_{X_k})$. We can define a map $g\in \Aut(X, T^2)$ 
    as $g(x)=g_k(x)$ for all $x\in X_k$ and all $k\in\{0,1,...,2^n-1\}$. Notice that since $\overline{0}$ is a fixed point of $T^{2^n}$ this map is well defined on the intersection. Moreover, since the sets ${X_0, X_1,...,X_{2^n-1}}$ are closed, $g$ in continuous and since $g_k$ commutes with $T^{2^n}$ for all $k\in\{0,1,...,2^n-1\}$, so does $g$. Since $(X_0, T^{2^n}|_{X_0})$ and $(X_k, T^{2^n})$ are conjugate for $k=0,1,...,2^{n}-1$. Hence, we have shown that there exists an injective map $\Phi\colon\Aut(X_0, T^2|_{X_0})^{2^n}\to \Aut(X, T)$. Additionally, the following sequence is exact. 
\begin{equation*}
\begin{tikzcd}[column sep=3ex]
1 \arrow[r] & \Aut(X_0, T^{2^n}|_{X_0})^{2
^n} \arrow[r, "\Phi"]  & { \Aut(X,T^{2^n})} \arrow[r, "\iota"]  & \Sym(2^n) \arrow[r] \arrow[l, dotted, bend right=-33, "s"] & 1 
\end{tikzcd}            
\end{equation*} 
This sequence splits as the map $s\colon\Sym(2^n)\to\Aut(X,T^{2^n})$ as $s(\pi)=\phi_\pi$ is a group homomorphism such that $\iota\circ s(\pi)=\pi$ for all $\pi\in \Sym(2^n)$. We conclude $\Aut(X,T^{2^n})\cong \Aut(X_0, T^{2^n}|_{X_0})^{2
^n}\rtimes \Sym(2^n)$. It is not hard to verify that this semi-direct product is in fact a wreath product, for more details see Proposition 2.3 of \cite{jones}.

\subsubsection{Example 2:}\label{example2} This is an example of a transitive system $(X,T)$ that has a cyclic almost-partition of size $3^n$ for all $n>0$ but whose stabilized automorphism group consists only of powers of $T$. For this example, we use the notion of a Sturmian subshift. To read an in depth presentation of Sturmian subshifts we refer the reader to Chapter 32 of \cite{fogg}, for example.

 %Let $F$ be a two sided Sturmian sequence of slope $1/\phi$ where $\phi$ is the golden ratio. This is commonly referred to as the {\em{bi-infinite Fibonacci word}}. 

 Let $(\hat F, T)$ be a Sturmian subshift given by the two sided Sturmian sequence $F=(x_i)_{i\in \Z}$.  By \cite{olli}, $\Aut(\hat F, T)=\langle T\rangle\cong\Z$. Using this fact, in \cite{stab}, it was proven that $\StabAut(\hat F, T)=\langle T\rangle\cong\Z$.

We proceed with the construction of the example. Take $A_0=\alpha\alpha\alpha$, where $\alpha$ denotes a symbol different from $0$ and $1$. Define $b_n=\alpha F_{[1,3^{n}-2]}\alpha$ for all $n\geq 1$ and $A_{n+1}=A_nb_nA_n$, for $n>1$. Notice $|A_n|=3^{n+1}$ for all $n\geq 1$. Define $\hat A$ to be the limit of this iterative process and $u$ to be the bi-infinite sequence given by $u_{(-\infty ,-1]}=F_{(-\infty, -1]}$ and $u_{[0,\infty)}=\hat A$. Let $X=\overline{\mathcal{O}_T(u)}$. By construction, $(X,T)$ is a transitive system. Moreover, notice that $\hat F\subseteq X$.

For all $n>1$, using the recursive relation, each $A_k$ with $k>n$  can be written as a sequence of elements in $\{A_n\}\bigcup \{b_i: n\leq i \leq k\}$ with no consecutive instances of $A_n$. By analyzing the positions in which the symbol $\alpha$ occurs in $A_k$ and since there are no consecutive appearances of $A_n$, we can see that all occurrences of $b_n$ in $A_k$ occur at the same integer mod $3^n$. Hence, each element in $\mathcal{O}(u)$ belongs to one and only one of the following sets
$$X_k=\{x\in X: \text{if } x_{[i, i+3^{n}-1]}=b_n \text{ then } i= k\pmod{ 3^{n}}\}, \quad \text{ for } k=0,1,...,3^n-1.$$ In the same manner as in the previous example, we can approximate each element in $X$ by elements in $\mathcal{O}(u)$ to conclude that $X=\bigcup_{i=0}^{3^n-1}X_i$ and $\bigcap_{i=0}^{3^n-1}X_i=\hat F$.

For a fixed $n>0$, define the following subsets of $X$ 
$$R=\{(y_i)\in X: \text{ there exist } N\in \Z \text{ such that }y_N=\alpha \text{ and }y_i\neq \alpha,\, \forall i> N \}$$
$$L=\{(y_i)\in X: \text{ there exist } N\in \Z \text{ such that }y_N=\alpha \text{ and }y_i\neq \alpha,\, \forall i< N \}$$
$$\Gamma=\{y\in X: y \text{ can be written as a sequence in }\{A_nb_i: i\geq n\}\}$$

Similarly to the methods used before and in the previous example, one can show that $X=R\cup L\cup \Gamma\cup\hat F$. Moreover, the sets $A=R\cup L$, $\Gamma$ and $\hat F$ are pairwise disjoint. It can easily be shown that $(\hat F, T|_{\hat F})$ is the only minimal subsystem of $(X,T)$. Hence, for any $\phi\in \Aut(X,T)$, we have that $\phi$ leaves $\hat F$ invariant, that is, $\phi(\hat F)=\hat F$. Hence, $\phi|_{\hat F}\in \Aut(\hat F, T)$. By the Stabilized Curtis-Hedlund-Lyndon Theorem from \cite{stab} and since $\Aut(\hat F, T)=\langle T \rangle$, $\phi$ acts like the shift on blocks of length $2 r_{\phi}+1$ that do not contain the symbol $\alpha$. This shows that $\phi$ leaves the sets $R$ and $L$ invariant. Since $\phi$ is a bijection, we can conclude $\phi$ maps $\Gamma$ to itself. 

Let $u=(u_i)_{i\in\Z}\in \Gamma$, then for a sequence $(i_j)_{j\in \Z}$, we can rewrite $u=(A_nb_{i_j})_{j_\in \Z}$. By picking $n$ large enough, much larger than $r_{\phi}$, there is a large portion of each $b_{i_j}$ block that is mapped to a shifted copy of itself in $\phi(u)$. Additionally, the small pieces of each $A_n$ that correspond to blocks $b_\ell$ for $\ell$'s that are large enough have sections of themselves that are also mapped to shifted copies of themselves in $\phi(u)$. Since $\phi(u)$ belongs to $\Gamma$, this allows us identify the positions of the blocks $A_n$. Which in turn allows us to identify the coordinates where the blocks $b_{i_j}$ begin and end. With these observations, the image of $u$ under $\phi$ is uniquely determined and it corresponds to a shift of $u$. We illustrate this process in the following picture where the red segments correspond to sections of potential ambiguity that can be deduced by the process explained above.
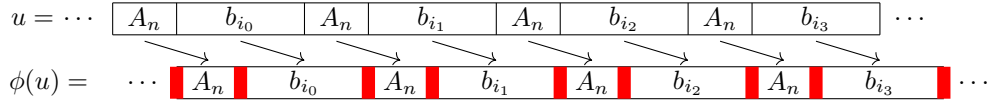
\begin{center}
    \begin{tikzpicture}[scale=0.85]
 \draw (0,0.5) -- (12,0.5);
  \draw (0,0) -- (12,0);
   \draw (0,0) -- (0,0.5);
    \draw (1,0) -- (1,0.5);
    \draw (3,0) -- (3,0.5);
    \draw (4,0) -- (4,0.5);
    \draw (6,0) -- (6,0.5);
    \draw (7,0) -- (7,0.5);
    \draw (9,0) -- (9,0.5);
    \draw (10,0) -- (10,0.5);
    \draw (12,0) -- (12,0.5);

    \draw (1,-0.5) -- (13,-0.5);
    \draw (1,-1) -- (13,-1);
    
    \draw [line width=5pt, red] (1,-1) -- (1,-0.5);
    \draw [line width=5pt, red] (2,-1) -- (2,-0.5);
    \draw [line width=5pt, red] (4,-1) -- (4,-0.5);
    \draw [line width=5pt, red](5,-1) -- (5,-0.5);
    \draw [line width=5pt, red] (7,-1) -- (7,-0.5);
    \draw [line width=5pt, red](8,-1) -- (8,-0.5);
    \draw [line width=5pt, red](10,-1) -- (10,-0.5);
    \draw [line width=5pt, red](11,-1) -- (11,-0.5);
    \draw[line width=5pt, red]  (13,-1) -- (13,-0.5);

    \draw [->] (0.5,-0.1) -- (1.5,-0.4);
    \draw [->] (2,-0.1) -- (3,-0.4);
    \draw [->] (3.5,-0.1) -- (4.5,-0.4);
    \draw [->] (5,-0.1) -- (6,-0.4);
    \draw [->] (6.5,-0.1) -- (7.5,-0.4);
    \draw [->] (8,-0.1) -- (9,-0.4);
    \draw [->] (9.5,-0.1) -- (10.5,-0.4);
    \draw [->] (11,-0.1) -- (12,-0.4);

    \draw (0.5,0.25) node {$A_n$};
    \draw (2,0.25) node {$b_{i_0}$};
    \draw (3.5,0.25) node {$A_n$};
    \draw (5,0.25) node {$b_{i_1}$};
   \draw (6.5,0.25) node {$A_n$};
   \draw (8,0.25) node {$b_{i_2}$};
    \draw (9.5,0.25) node {$A_n$};
    \draw (11,0.25) node {$b_{i_3}$};

    \draw (1.5,-0.75) node {$A_n$};
    \draw (3,-0.75) node {$b_{i_0}$};
    \draw (4.5,-0.75) node {$A_n$};
    \draw (6,-0.75) node {$b_{i_1}$};
   \draw (7.5,-0.75) node {$A_n$};
   \draw (9,-0.75) node {$b_{i_2}$};
    \draw (10.5,-0.75) node {$A_n$};
    \draw (12,-0.75) node {$b_{i_3}$};

    \draw (-0.5,0.25) node {$\dots$};
    \draw (12.5,0.25) node {$\dots$};
    \draw (0.5,-0.75) node {$\dots$};
    \draw (13.5,-0.75) node {$\dots$};
    
    \draw (-1.25,0.25) node {$u=$};
    \draw (-1,-0.75) node {$\phi(u)=$};

\end{tikzpicture}
\end{center}

Since every finite word in $X$ appears as a subword of an element in $\Gamma$, by Stabilized Curtis-Hedlund-Lyndon Theorem from \cite{stab}, we conclude that $\phi=T^k$ for some $k\in \Z$. This proves $\StabAut(X,T)=\langle T \rangle$.

% Let $g\in \StabAut(X, T)$. Since $g|_{\hat F}\in \StabAut(\hat F, T)$ we have that for all $x\in \hat F$, $g(x)=T^j(x)$, for some $j\in \Z$. Take $$\Lambda=\{ x=(x_i) \in X: x_{j+\ell}=\hat F_\ell\text{ for some } j\in \Z \text{ and for all } \ell>0\}.$$ Lemma 3.2 of \cite{stab} and the previous remark imply that for all $x\in \Lambda$, $x$ is asymptotic to $T^{-j}(g(x))$ for some $j\in \Z$. Applying Lemma 2.3 of \cite{ddmp} we have that $g(x)=T^j(x)$ for all $x\in \Lambda$. Since there exists a sequence of elements in $\Lambda$ that converges to $u$ and $u$ is a transitive point we conclude that $g(x)=T^j(x)$ for all $x\in X$. This shows that $\StabAut(X,T)=\langle T\rangle$.

% % % % % % % % % % SECTION % % % % % % % % % %
\section{Systems with a finite number of asymptotic classes} \label{section:FiniteAS}

For any system $(X,T)$, we say that two points $x$ and $y$ in $X$ are \textit{asymptotic} if $d(T^nx,T^ny)\to 0$ as $n\to \infty$. Define the equivalence relation $\sim$ on $X$ where $x\sim y$ if and only if $x$ is asymptotic to $T^ky$ for some $k\in \Z$. If $x\sim y$, then ${x,y}$ is called an \textit{asymptotic pair}, and the equivalence classes under $\sim$ (that are not just one orbit) are called the \textit{asymptotic components} of $(X,T)$. The set of all asymptotic components is denoted by $\mathcal{AS}(X,T)$.

In this section, we study the stabilized automorphism group of systems with a finite number of asymptotic components. This constitutes a large class of systems that includes many important sub-classes, such as systems of finite topological rank (for more information on this class, refer to \cite{DownMaass} or \cite{EspinozaMaass}). This subclass, in turn, includes subshifts with non-superlinear complexity, {\em i.e.}, systems satisfying
$\liminf_{n\to \infty}p_X(n)/n<\infty$ (for further details on this class, see \cite{CyrKralinear}, \cite{ddmp}).

We begin the study of the stabilized automorphism of this class of systems by proving the following two lemmas. 

\begin{lemma}\label{point_trans_for_power}
Let $(X,T)$ be a system with a transitive point $\hat{x}$.
If $n \geq 1$ is such that $T^n$ acts transitively on $X$, then $\hat{x}$ is transitive for $(X,T^n)$.
\end{lemma}
\begin{proof}
Let $Y$ be the closure of the $T^n$-orbit of $\hat{x}$.
Note that $Y \cup T(Y) \cup \dots \cup T^{n-1}(Y)$ is equal to the $T$-orbit of $\hat{x}$, which is equal to $X$.
So, $Y$ has non-empty interior.
Being $Y$ a closed $T^n$-invariant set, we deduce from the transitivity of $T^n$ that $Y = X$, that is, that $\hat{x}$ is transitive for $T^n$.
\end{proof}

\begin{lemma}\label{bound}
Let $(X,T)$ be a subshift with $K$ asymptotic components.
Assume that there is a  transitive point $\hat{x} \in X$ that is asymptotic to a different point. 
If $n \geq 1$ is such that $T^n$ acts transitively on $X$, then $|\Aut(T^n) / \langle T\rangle| \leq K$.
\end{lemma}

\begin{proof}
Let $g_0,\dots,g_K \in \mathrm{Aut}(X,T^n)$ be arbitrary.
It is enough to prove that $g_i \in \langle T\rangle g_j$ for some $i,j \in \{0,\dots,K\}$ with $i \neq j$.

Let $\cC$ be the set of asymptotic components of $(X,T)$. 
It is possible to find a set $\cC_0 \subseteq X$ such that:
\begin{enumerate}
\item[(i)] for every $C$ in $\cC$, there is an element in $\cC_0$ belonging to $C$;
\item[(ii)] for every $x$ in $\cC_0$ there is exactly one asymptotic components to which $x$ belongs; and
\item[(iii)] $\hat{x} \in \cC_0$.
\end{enumerate}
By (iii), there is $\hat{y} \in X \setminus \{\hat{x}\}$ such that $(\hat{x},\hat{y})$ is asymptotic.
Observe that $(g_i(\hat x),g_i(\hat y))$ is an asymptotic pair of $(X,T)$ for all $i \in \{0,\dots,K\}$.
Hence, by (i), there is $x_i \in \cC_0$ and $\ell_i \in \Z$ such that $(T^{\ell_i} g_i(\hat x), x_i)$ is asymptotic in $(X,T)$.
Now, (ii) implies that $|\cC_0| \leq |\cC| = K$, so we can use the Pigeonhole principle to find $i,j \in \{0,\dots,K\}$, with $i \not= j$, such that $x_i = x_j$.
Then, $(T^{\ell_i} g_i(\hat x), T^{\ell_j} g_j( \hat x))$ is asymptotic in $(X,T)$.
In particular, for any accumulation point $z$ of $(T^{kn} \hat x)_{k\geq0}$, $T^{\ell_i} g_i(z) = T^{\ell_j} g_j(z)$.
Since $\hat x$ is a transitive point for $T^n$, $T^{\ell_j} g_i = T^{\ell_i} g_j$. 
Therefore, $g_i \in \langle T\rangle g_j$ and the proposition follows.
\end{proof}

We now use some notation introduced before \autoref{aut_as_eig&min_parts}. Recall that $\lcm(n,n') \in M(T)$ for any $n,n' \in M(T)$.

\begin{lemma} \label{bound_MinAut} % this lemma could be merged with lemma {bound}.
Let $(X,T)$ be a subshift with finitely many asymptotic components.
Assume that $\hat{x} \in X$ is a transitive point that is asymptotic to a different point.
Then, there exists $N \geq 1$ such that 
$$ \Aut_{M(T)}(T) = \Aut(T^N). $$
\end{lemma}
\begin{proof}
Let $(p_k)_{k \geq 1}$ be an enumeration of the elements of $M(T)$.
We set $q_k = \lcm(p_1,\dots,p_k) \in M(T)$.
Then, for every $n \in M(T)$, there exists $k \geq 1$ such that $\Aut(T^n) \subseteq \Aut(T^{q_k})$, that is,
\begin{equation*}
    \Aut_{M(T)}(T) = \bigcup_{k\geq1} \Aut(T^{q_k}).
\end{equation*}
Additionally, $\Aut(T^{q_k}) \subseteq \Aut(T^{q_{k+1}})$ for all $k\geq1$. 
Hence, we have the following chain of inclusions of groups 
\begin{center}
\begin{tikzcd}[column sep=1.5em]
\Aut(T^{q_1})/\langle T\rangle
\arrow[r, hook] & \Aut(T^{q_2})/\langle T\rangle 
\arrow[r, hook] & \ \ldots \ 
\arrow[r, hook] & \Aut(T^{q_k})/\langle
\arrow[r, hook] T\rangle & \ \ldots
\end{tikzcd}
\end{center}
By \autoref{bound}, the groups on the previous chain have order bounded by $|\mathcal{AS}(T)|$, which is finite.
Therefore, there exists $k \geq 1$ such that 
\begin{equation*}
    \Aut(T^{q_k})/\langle T\rangle = \Aut(T^{q_\ell})/\langle T\rangle
    \ \text{for all $\ell \geq k$.}
\end{equation*}
This implies that $\Aut_{M(T)}(T) / \langle T \rangle = \Aut(T^{q_k})\langle T \rangle$, and thus that $\Aut_{M(T)}(T) = \Aut(T^{q_k})$.
\end{proof}

%{\color{blue} Extract the following theorem from Theorems \ref{semidirectproduct} and \ref{aut_as_eig&min_parts}}
\begin{theorem} \label{theo:aut-if-finiteAsymptotic}
Let $(X,T)$ be a minimal subshift with finitely many asymptotic components, and let $N\geq 1$ be as in \autoref{bound_MinAut}. 
Then,
\begin{equation} \label{eq:aut-if-finiteAsymptotic:1}
    \Aut^\infty(X,T) = \bigcup_{m\in \Eig(T)} \Aut(X,T^{m\cdot N}).
\end{equation}
In particular, the stabilized automorphism group of $(X,T)$ is isomorphic to the direct limit
\begin{equation*} \label{eq:aut-if-finiteAsymptotic:2}
    \varinjlim_{m \in \Eig(T)}
    \Aut(T^{m\cdot N}|_{X_m}) \wr \Sym(m).
\end{equation*}
Moreover, $\Aut(X_{m}, T^{m\cdot N}|_{X_m}) / \langle T^{m\cdot N}|_{X_m} \rangle$ is isomorphic to $\Aut(X, T^N) / \langle T^N \rangle$.
\end{theorem}
\begin{proof}
Let $N$ be given by \autoref{bound_MinAut} and let $n\geq1$. 
We can decompose, using \autoref{decompose:trans_peig}, $n = m \cdot k$, with $m \in \CAP(T)$ and $T^k$ acting transitively on $X$.
Observe that $\ell \in \Eig(T)$ by \autoref{minimal_sofic=>cP} and that $k \in M(T)$ by the definition of $M(T)$.
Hence, by property defining $N$, $\Aut(T^n)$ is a subgroup of $\Aut(T^{N\cdot\ell})$. 
This proves \eqref{eq:aut-if-finiteAsymptotic:1}.
The rest of the theorem then follows from \autoref{semidirectproduct}.
\end{proof}

\begin{remark} \label{moregeneralversion}
In \autoref{theo:aut-if-finiteAsymptotic}, the minimality of $(X,T)$ can be weakened to satisfying Condition $\cP$ and having a transitive point that is asymptotic to a different point.
\end{remark}

\begin{corollary}
Let $(X,T)$ be a subshift that satisfies the hypothesis of the last theorem (alternatively, satisfying the hypothesis of \autoref{moregeneralversion}). 
If $(X,T)$ has finitely many rational eigenvalues, then $\StabAut(T)$ is virtually $\Z^d$, where $d$ is the greatest element in $\Eig(T)$.
\end{corollary}

% % % % % % % % % % SECTION % % % % % % % % % %
\section{Irreducible subshifts of finite type}\label{section:irreducibleSFT}

In \cite{Pent}, Schmieding proved the that if $(X, \sigma_X)$, $(Y, \sigma_Y )$ are non-trivial mixing shifts of finite type with isomorphic stabilized automorphism groups, then $h_{\mathrm{top}}(\sigma_X) / h_{\mathrm{top}}(\sigma_Y) \in \Q$.
Making use of \autoref{semidirectproduct} and \autoref{prop:main_alg_lemma}, we are able to extend Schmieding's methods to cover irreducible subshifts of finite type.

\begin{proof}[Proof of \autoref{main:irreducible_sofic_case}] 
% página 31, (6.2) (SPECTRAL DECOMPOSITION OF DIFFEOMORPHISMS)
By the Smale spectral decomposition \cite{Smale}, there exist $m, k \geq 1$ and clopen sets $X_m\subseteq X$ and $Y_k\subseteq Y$ such that $\{\sigma_X^i(X_m): 0\leq i<m\}$ and $\{\sigma_Y^i(Y_k): 0\leq i<k\}$ form disjoint clopen partitions of $X$ and $Y$, respectively. 
Moreover, $(X_m, \sigma_X^m|_{X_m})$ and $(Y_k, \sigma^k_Y|_{Y_k})$ are mixing. 
Since any mixing system has no eigenvalues other than $1$, we have that $m$ and $k$ are the largest elements of $\Eig(\sigma_X)$ and $\Eig(\sigma_Y)$ respectively. 
If $m = 1$ or $k = 1$, then $m = k = 1$ and we conclude by the main result of Schmieding's paper \cite{Pent}.
In what follows, we assume that $m ,k \geq 4$; see \autoref{rem:scott-generalization} for the considerations that have to be taken in the general case.

We have, as $\Eig(\sigma_X) = \CAP(\sigma_X)$ and $\Eig(\sigma_Y) = \CAP(\sigma_Y)$ by \autoref{minimal_sofic=>cP}, that \autoref{aut_as_eig&min_parts} can be applied, yielding
\begin{align*}
\StabAut(\sigma_X)   &\cong
\Aut_{M(\sigma_X)}(\sigma_X^m|_{X_m}) \wr \Sym(m)   \\ &
\cong   \Aut_{M(\sigma_Y)}(\sigma_Y^k|_{Y_k}) \wr \Sym(k)  
\cong  \StabAut(\sigma_Y).
\end{align*}
Being $m, k \geq 4$, we can use \autoref{prop:main_alg_lemma}, which yields $m=k$ and 
$$\Aut_{M(\sigma_X)}(\sigma_X^m|_{X_m}) \cong \Aut_{M(\sigma_Y)}(\sigma_Y^k|_{Y_k}).$$
Now, since $(X_m, \sigma_X^m|_{X_m})$ and $(Y_k, \sigma^k_Y|_{Y_k})$ are mixing, we can use the main result in \cite{Pent} to obtain that
$$\frac{h_{\mathrm{top}}(\sigma_X^m|_{X_m})}{h_{\mathrm{top}}(\sigma^k_Y|_{Y_k})}\in \Q.$$
As $h_{\mathrm{top}}(\sigma_X^m|_{X_m}) = m\cdot h_{\mathrm{top}}(\sigma_X)$ and $h_{\mathrm{top}}(\sigma_Y^k|_{Y_k}) = k\cdot h_{\mathrm{top}}(\sigma_Y)$, the theorem follows. 
\end{proof}

\begin{remark} \label{rem:scott-generalization}
The case in which $m \in \{2,3\}$ or $k \in \{2,3\}$ can be treated as follows.
The remark that follows \autoref{prop:main_alg_lemma} affirms that one of the groups $\StabAut(\sigma_X^m|_{X_m})$ and $\StabAut(\sigma_Y^k|_{Y_k})$ is a subgroup of index at most 2 of the other.
Thus, the technique of \cite{Pent} can be used to prove that $h_{\mathrm{top}}(\sigma_X^m|_{X_m}) / h_{\mathrm{top}}(\sigma_Y^k|_{Y_k}) \in \Q$. 
This implies that the conclusion of \autoref{main:irreducible_sofic_case} holds in this case.
\end{remark}

% % % % % % % % % % SECTION % % % % % % % % % %

\printbibliography

\end{document}